\documentclass{article}
\usepackage[utf8]{inputenc}
\usepackage{amssymb,amsthm,amsfonts,graphics}
\usepackage{fullpage}
\usepackage{rotating}
\usepackage{stmaryrd}
\usepackage{tikz}
\usepackage{verbatim}
\usepackage{url}
\usepackage{multicol}
\usepackage{xr-hyper}
\usepackage{mathrsfs}
\usepackage{float}

\theoremstyle{plain}
\newtheorem{theorem}{Theorem}
\newtheorem{lemma}[theorem]{Lemma}

\newtheorem{corollary}[theorem]{Corollary}

\newtheorem{conjecture}[theorem]{Conjecture}

\makeatletter
\newtheorem*{rep@theorem}{\rep@title}
\newcommand{\newreptheorem}[2]{%
\newenvironment{rep#1}[1]{%
 \def\rep@title{#2 \ref{##1}}%
 \begin{rep@theorem}}%
 {\end{rep@theorem}}}
\makeatother

\newreptheorem{theorem}{Theorem}
\newreptheorem{lemma}{Lemma}
\newreptheorem{corollary}{Corollary}

\newcommand{\N}{\mathbb{N}}

\newcommand{\R}{\mathbb{R}}

\newcommand{\pr}{\mathbb{P}}
\newcommand{\E}{\mathbb{E}}
\newcommand{\Po}{\rm Po}
\newcommand{\Bin}{\rm Bin}

\newcommand{\ind}{\mathbf 1}

\usepackage[tbtags]{amsmath}
\usepackage[cal=boondoxo]{mathalfa}

\newcommand{\cc}[1]{{\color{black}#1}}

\newcommand{\hered}{\mbox{\rm Hered}}
\newcommand{\minor}{\mathrm{Minor}}

\newcommand{\pend}{\mbox{\rm pend}}
\newcommand{\whp}{\mbox{\rm whp}}
\newcommand{\wvhp}{\mbox{\rm wvhp}}
\newcommand{\comp}{{\mathrm Comp}}
\newcommand{\con}{\mbox{\rm Conn}}

\newcommand{\cA}{\mathcal A}
\newcommand{\cB}{\mathcal B}
\newcommand{\cC}{\mathcal C}
\newcommand{\cD}{\mathcal D}
\newcommand{\cE}{\mathcal E}
\newcommand{\cF}{\mathcal F}
\newcommand{\cG}{\mathcal G}
\newcommand{\cH}{\mathcal H}
\newcommand{\cL}{\mathcal L}
\newcommand{\cP}{\mathcal P}
\newcommand{\cO}{\mathcal O}
\newcommand{\cN}{\mathcal N}

\newcommand{\cS}{\mathcal S}

\newcommand{\tA}{\widetilde{\mathcal A}}
\newcommand{\tB}{\widetilde{\mathcal B}}
\newcommand{\tC}{\widetilde{\mathcal C}}
\newcommand{\tD}{\widetilde{\mathcal D}}

\newcommand{\tP}{\widetilde{\mathcal P}}

\newcommand{\Frag}{{\rm Frag}}
\newcommand{\frag}{{\rm frag}}

\newcommand{\m}{}
\newcommand{\inu}{\in_u}

\newcommand{\fsgr}{{\rm fsgr}}

\newcommand{\aut}{{\rm aut}}

\newcommand{\bS}{\mathbf S}
\newcommand{\eps}{\varepsilon}

\title{Random graphs embeddable in order-dependent 
surfaces}
\author{
Colin McDiarmid\\Department of Statistics\\ University of Oxford\\
cmcd@stats.ox.ac.uk\\
\and
Sophia Saller\\
Department of Mathematics\\ University of Oxford\\
and DFKI\\
sophia.saller@dfki.de}
\date{\today}

\begin{document}

\maketitle

\begin{abstract}
Given a `genus' function $g=g(n)$, we let $\cE^g$ be the class of all graphs $G$ such that if $G$ has order $n$ (that is, has $n$ vertices) then it is embeddable in a surface of Euler genus at most $g(n)$.
Let the random graph $R_n$ be sampled uniformly from the graphs in $\cE^g$ on vertex set $[n]=\{1,\ldots,n\}$.  Observe that if $g(n)$ is 0 then $R_n$ is a random planar graph, and if $g(n)$ is sufficiently large then $R_n$ is a binomial random graph $G(n,\tfrac12)$.
We investigate typical properties of $R_n$. We find that for \emph{every} genus function $g$, with high probability at most one component of $R_n$ is non-planar.  In contrast, we find a transition for example for connectivity: if
$g$ is non-decreasing and $g(n) = O(n/\log n)$ then $\liminf_{n \to \infty} \pr(R_n \mbox{ is connected}) < 1$, and if $g(n) \gg n$ then with high probability $R_n$ is connected. These results also hold when we consider orientable and non-orientable surfaces separately. We also investigate random graphs sampled uniformly from the `hereditary part' or the `minor-closed' part of $\cE^g$, and briefly consider corresponding results for unlabelled graphs.  
\end{abstract}


\section{Introduction}

Given a surface $S$, let $\cE^S$ be the class of all (finite, simple) graphs embeddable in $S$ (not necessarily cellularly), so the class $\cP$ of planar graphs is $\cE^{\bS_0}$ where $\bS_0$ is the sphere. For any given class $\cA$ of graphs, for each $n \in \N$ we let $\cA_n$
denote the set of graphs in $\cA$ on vertex set $[n]: = \{1,\ldots,n\}$, and let $R_n \inu \cA $ indicate that the random graph $R_n$ is chosen uniformly at random from $\cA_n$ (assuming this set is non-empty).
A \emph{genus function} is a function $g=g(n)$ from the positive integers $\mathbb{N}$ to the non-negative integers $\mathbb{N}_0$: we shall always take $g$ to be such a function. We let $\cE^g$ be the class of all graphs $G \in \cE^S$ for some surface $S$ of Euler genus at most $g(n)$ where $n=v(G)$ (the number of vertices of $G$).
If we insist that all the surfaces $S$ involved are orientable we obtain the graph class $\cO\cE^g$, and similarly if we insist that the surfaces are non-orientable we obtain $\cN\cE^g$ (where $\cN\cE^0$ is taken to be $\cP$). 
For a full discussion of embeddings in a surface see for example~\cite{GraphsonSurfaces}.
In the following, let the class $\cA^g$ be any one of $\cE^g$, $\cO\cE^g$, $\cN\cE^g$ or $\cO\cE^g\cap \cN\cE^g$. We consider $R_n \inu \cA^g$ and are interested in properties of $R_n$, such as the probability that $R_n$ is connected, the typical number of edges or faces (in a relevant embedding), 
and so on.

The class $\cP$ of planar graphs, and more generally the class $\cE^S$
of graphs embeddable in a fixed surface $S$, have received much attention recently.
Much is known about typical properties of $R_n \inu \cP$,
see for example \cite{RandCub, PlanGraphsViaWell, TherandomPlanar, QuadraticExact, EdgesRand, RandPlanNnod, RandPlanAvgDeg, Asymptoticformula, LabGC, GCgiven3, MaxDegree, RandomPlanar, addable, RandTriang}.
The corresponding questions for $R_n \inu \cE^S$
have also been extensively studied and much is known, see for example \cite{asymptLabGraphs, LimitLawsFixedS, evolutionRandomS, EvolutionGiant, PhaseTransitions, EnumerationCubicSurf, RandomGraphs}.
For some of these results,
all that is needed for their proof is the existence of a growth constant for the class or even just a positive radius of convergence of the exponential generating function.
We say that a class $\cA$ of (labelled) graphs has \emph{(labelled) growth constant} $\gamma\,$ if $0<\gamma<\infty$ and
\begin{equation} 
    \left( \left|\cA_n \right| / n! \right)^{\frac1{n}} \rightarrow \gamma \;\; \mbox{ as } n \rightarrow \infty \, .\notag
\end{equation}
Given a class $\cB$ of (labelled) graphs, we let $\rho(\cB)$ be the radius of convergence of the exponential generating function $B(x) = \sum_n |\cB_n|/n! \: x^n$, so $\rho(\cB) = \left(  \limsup_{n\rightarrow \infty} \left( |\cB_n| / n! \right)^{\frac1{n}}\right)^{-1}$.  Thus for example $\rho(\cP)= \gamma_{\cP}^{-1}$. 
For any class $\cB$ of (labelled) graphs we denote by $\tB$ the corresponding set of unlabelled graphs; and we let $\tilde{\rho}(\tB)$ be the radius of convergence of the ordinary generating function $\tilde{B}(x) = \sum_n |\tB_n| \: x^n$, so $\tilde{\rho}(\tB) = \left(  \limsup_{n\rightarrow \infty} |\tB_n|^{\frac1{n}}\right)^{-1}$.
In a companion paper \cite{MSsizes}, we have shown that the class $\cA^g$ has a growth constant when $g(n) = o(n/\log^3 n )$, and that $\rho(\cA^g)>0$ if and only if $g(n)=O(n/\log n)$. (When $\log$ has no subscript we take it to mean the natural logarithm.) From the existence of a growth constant we can deduce some common properties for $R_n \inu \cA^g$ 
when we have relatively small genus $g$. But what if
the genus is as large as $n/\log^3 n$ or even $n/\log n$? In the current paper we give results on $R_n \inu \cA^g$
for a wide range of functions $g$.

This paper is part of an overarching project in which we investigate two closely related questions : (a) how large is the graph class $\cA^g$; and (b) what are typical properties of a random $n$-vertex graph $R_n$ sampled uniformly from such a class? In the companion paper \cite{MSsizes}, we investigate question (a) and give estimates and bounds on the sizes of these classes of graphs. In the current paper, we consider question~(b) on random graphs, using results from~\cite{MSsizes}. A central aim
in both of these papers is to determine where there is a change (`phase transition') between `planar-like' behaviour and behaviour like that of a binomial (or Erd\H{o}s-R\'enyi) random graph, both for class sizes and for properties. See~\cite{DKMS2019} for results on the evolution of random graphs on non-constant orientable surfaces when we consider also the number of edges.

We say that a sequence $A_n$ of events occurs \emph{with high probability} ($\whp$) if $\mathbb{P}(A_n)\rightarrow 1$ as $n \rightarrow \infty$; and we say that $A_n$ occurs \emph{with very high probability} ($\wvhp$) if $\mathbb{P}(A_n) = 1 - e^{-\Omega(n)}$ as $n \rightarrow \infty$.


\section{Statement of Results}
\label{sec.main}

Recall that $g=g(n)$ is a genus function. It may be natural to assume that $g$ is non-decreasing, and sometimes we will do so. Recall also that $\cA^g$ denotes one of the classes of graphs $\cO\cE^g$, $\cN\cE^g$, $\cE^g = \cO\cE^g \cup \cN\cE^g$ or $\cO\cE^g \cap \cN\cE^g$.

We first consider random embeddable graphs $R_n \inu \cA^g$,  where we insist simply that the graph be embeddable in the appropriate surface (of Euler genus at most $g(n)$ for an $n$-vertex graph) and we have no other requirements.  Our main focus is on this case, presented in
Section~\ref{subsec.embed}. After considering such random graphs, in Section~\ref{subsec.heredembed} we consider random `hereditarily embeddable' graphs, where we insist also that each  induced subgraph is embeddable in an appropriate surface, depending on its number of vertices. 
We briefly consider random unlabelled graphs from $\tilde{\cA}^g$ in Section~\ref{subsec.unlab}. Finally, we conclude with a brief plan of the rest of the paper in Section~\ref{subsec.plan}.

\subsection{Random graphs from an embeddable class, $R_n \inu \cA^g$}
\label{subsec.embed}

We present six theorems on random graphs embeddable in  surfaces, concerning: planarity of the fragment; connectivity; maximum degree; maximum face size; and numbers of edges, leaves and faces (not in this order).  Throughout Section~\ref{subsec.embed} we let $R_n \inu \cA^g$.

First we see that the fragment of a random graph $R_n \inu \cA^g$ is likely to be planar for \emph{every} genus function~$g$. Recall that the \emph{fragment} of a graph $G$ is the (unlabelled) subgraph induced on the vertices not in the largest (by number of vertices) component
(where say we break ties lexicographically). We denote this graph by $\Frag(G)$ and its number of vertices 
by $\frag(G)$.
\m{Lemma 30 of sizes concerns max `component-degree'.  Also true here}

\begin{theorem} \label{thm.Fragplanar}
For every genus function $g$, $\,\Frag(R_n)$ is planar whp.
\end{theorem}

By precise asymptotic estimates for $|\cP|$, $|\cO\cE^2_n|$ and $|\cN\cE^1_n|$, see  \cite{asymptLabGraphs,LimitLawsFixedS,Asymptoticformula}, we have $|\cO\cE^2_n| / |\cP_n| = \Theta(n^{5/2})$ and $|\cN\cE^1_n| / |\cP_n| = \Theta(n^{5/4})$.  Thus $|\cO\cE^2_n| \gg |\cP_n|$ (which also follows from equation~(\ref{eqn.lbasymp}) below) and $|\cN\cE^1_n| \gg |\cP_n|$.  Thus if $g(n) \geqslant 2$ for all sufficiently large $n$ and $R_n \inu \cO\cE^g$ then whp $R_n$ is non-planar, and so by Theorem~\ref{thm.Fragplanar} whp exactly one component of $R_n$ is non-planar. 
Similarly, if $g(n) \geqslant 1$ for all sufficiently large $n$ and $R_n \inu \cN\cE^g$ then whp exactly one component of $R_n$ is non-planar. 

\smallskip

A central interest here concerns connectivity.  Let $C(x)$ be the exponential generating function of the class $\cC$ of connected planar graphs, and let $\lambda = C(\rho(\cP)) \approx 0.037439$.
For $R_n \inu \cP$, the probability that $R_n$ is connected tends to the constant $p^* = e^{-\lambda} \approx  0.963$; and indeed $\kappa(R_n)-1$ is distributed asymptotically as a Poisson law of parameter $\lambda$.
Further, if $S$ is any fixed surface and $R^S_n \inu \cE^S$, then the same results hold for $\kappa(R^S_n)$ and the probability that $R^S_n$ is connected~\cite[Theorem 5.2]{LimitLawsFixedS}. See also~\cite[Corollary 1.6 c]{randomGraphsMinorClosed}, and Theorem~\ref{thm.heredfrag} and Section~\ref{subsec.BPPRG} below.


\begin{theorem}\label{thm.conn}
\begin{description}
\item{(a)}
If $g(n)$ is $o(n/\log^3 n)$ then $\,\limsup_{n \to \infty} \pr(R_n \mbox{ is connected}) \leqslant p^*$. 

\item{(b)}
If $g(n)$ is $O(n/\log n)$ and $g$ is non-decreasing, then  $\; \liminf_{n \to \infty} \,\pr(R_n \mbox{ is connected}) < 1$. 

\item{(c)} 
 If $\, g(n) \gg n \,$ then whp $R_n$ is connected.
\end{description}
\end{theorem}

\noindent
Conjecture 14 of \cite{MSsizes} concerns how, for a given surface $S$, $|\cE^S_n|$ grows when we increase $n$ by 1.
If this conjecture held then, when $\cA^g$ is $\cO\cE^g$ or $\cN\cE^g$ or $\cE^g$, we could improve the statement in part (a) above 
to say the following: if $g(n)$ is $o(n/\log^3 n)$ 
then 
$\pr(R_n \mbox{ is connected}) \to p^*$ as $n \to \infty$. (We justify this assertion immediately after the proof of Theorem~\ref{thm.conn}~(a) in Section~\ref{sec.conn} below.)

Consider part (b) of this theorem.
Observe that the conclusion $\, \liminf_{n \to \infty} \,\pr(R_n \mbox{ is connected}) < 1$ is equivalent to the statement that it is not the case that $R_n$ is connected whp. 
Perhaps if the genus function $g(n) = O(n/\log n)$ and $g$ is non-decreasing
then $\limsup_{n \to \infty} \,\pr(R_n \mbox{ is connected)} <1$? Indeed `for most $n$' this is true.
The \emph{lower (asymptotic) density} of a set $I \subseteq \N$ is defined to be $\liminf_{n \to \infty} |I \cap [n]| / n$.
We shall see in Lemma~\ref{lemma.connectedLD} that there is a set $I \subseteq \N$ of lower density arbitrarily close to 1 such that 
$\limsup_{n \to \infty, n \in I} \,\pr(R_n \mbox{ is connected)} <1$.
By parts (b) and (c) of Theorem~\ref{thm.conn}, the threshold for $R_n$ to be connected whp is when $g(n)$ is somewhere between about $n/\log n$ and about $n$. 

We are about to discuss numbers of edges, but for now let us note that each graph in $\cE^g_n$ has average degree less than $6(1\, +g(n)/n)$, see equation~(\ref{maxedges}) below.  Thus if we want $R_n$ to have growing average degree then we must have $g(n) \gg n$; and in this case $R_n$ is connected whp by part (c) above.  In contrast, for the binomial 
random graph $G(n,p)$ 
the threshold for connectivity is when the average degree is $\log n\,$ \cite{EvolutionRandom}.
\smallskip

In order to prove Theorem~\ref{thm.conn} on connectivity, we need to learn about the number $e(R_n)$ of edges of $R_n$ and the number $\ell(R_n)$ of leaves (where a leaf is a vertex of degree 1).
To prove the first result about the number of edges (Theorem \ref{thm.edges}(a)) we also need to know something about numbers of faces.  We next present our results on numbers of edges and faces, Theorems~\ref{thm.edges} and~\ref{thm.faces} and Corollary~\ref{cor.edgesfaces}.
\smallskip

We first consider briefly a fixed surface $S$ and $R^S_n \inu \cE^S$, where the behaviour of $e(R^S_n)$ is well understood~\cite{LabGC,LimitLawsFixedS}.  In this case $e(R^S_n)$ is asymptotically normal with mean $\sim \kappa n$ and variance $\sim \lambda n$, with known constants $\kappa \approx 2.21326$ and $\lambda \approx 0.43034$.
When the surface is not fixed we know far less.
Recall that, by Euler's formula, see equations~(\ref{eqn.Eulersformula}) and~(\ref{eqn.Euler2}), each graph in $\cE^g_n$ has less than $3(n+g)$ edges, see equation~(\ref{maxedges}).  
Theorem~\ref{thm.edges} gives lower bounds on $e(R_n)$, in part (a) for `most' genus functions $g$, and in part (b) for certain $g$ with $g(n) \gg n$.   The bound in (a) may be compared with equation~(\ref{eqn.egeqg}).

\begin{theorem}
\label{thm.edges}
\begin{description}
\item{(a)} If $g(n)$ is $o(n^2)$ then whp $e(R_n) > n+g$.
\m{\cc{We could extend (a) up to $g(n)$ nearly $\frac12 n^2$ - not now.  We could put $\frac1{12} n^2$ quickly using (b), but we could do better}}

\item{(b)}
Let $j \in \N$ and suppose that $\,n^{1+1/(j+1)}\ll g(n) \ll n^{1+1/j}$,
except that when $j=1$ we extend the range to
$n^{3/2} \ll g(n) \leqslant \tfrac1{12} n^2$.
Then whp $e(R_n) \geqslant (1+o(1))\, \tfrac{j+2}{j} \, g$.
\end{description}
\end{theorem}
Note that in the inequalities for $e(R_n)$ above we write $g$ rather than $g(n)$ for readability - we shall often do this.
Call an embedding of a graph $G \in \cA^g_n$ \emph{relevant} if the corresponding surface $S$ has Euler genus at most $g(n)$, and $S$ is orientable if $\cA$ is $\cO\cE$, $S$ is nonorientable if $\cA$ is $\cN\cE$, and so on. 
By Euler's formula, each relevant embedding of a graph in $\cA^g_n$ has at most $2(n+g)$ faces, see equation~(\ref{maxfaces}).
Theorem~\ref{thm.faces} gives lower bounds on numbers of faces, in part (a) for `most' genus functions 
$g$, and in part (b) for certain large $g$ as in part (b) of Theorem~\ref{thm.edges}.

\begin{theorem}
\label{thm.faces}
\begin{description}
\item{(a)} If $g(n)$ is $o(n^2)$, then whp the minimum number of faces in a relevant embedding of $R_n$ is at least $\frac{n^2}{15(n+g)}$.

\item{(b)} Let $j \in \N$ and suppose that $\,n^{1+1/(j+1)}\ll g(n) \ll n^{1+1/j}$, except that when $j=1$ we extend the range to
$n^{3/2} \ll g(n) \leqslant \tfrac1{12} n^2$.
Then whp the minimum number of faces in a relevant embedding of $R_n$ is at least $(1+o(1))\, \tfrac2{j}\, g$.
\end{description}
\end{theorem}

\begin{corollary}
\label{cor.edgesfaces}
If $n^{3/2}\ll g(n)\leqslant \tfrac1{12} n^2$, then whp $e(R_n) = (3+o(1))\,g$; and whp every relevant embedding of $R_n$ has $(2+o(1))\,g$ faces. 
\end{corollary}
\noindent
This corollary follows directly from Theorems \ref{thm.edges} (b) and \ref{thm.faces} (b) together with the upper bounds we noted before the theorems. It shows that for these genus functions, whp every relevant embedding of $R_n$ is a `near-triangulation',
with a $o(1)$ proportion of edges which are not in triangular faces, and thus a $o(1)$ proportion of faces which are not triangles.
\medskip

\noindent
Consider $R_n \in_u \cA^g$.
It follows directly from Theorem~\ref{thm.edges} that whp $e(R_n)= \Theta(n+g)$ for all $g$ such that $g(n) = O(n^2)$.
Also it follows directly from Theorem~\ref{thm.faces} that if $g(n)=O(n)$ then whp each relevant embedding of $R_n$
has $\Theta(n)$ faces, and if $g(n) \gg n$ then whp each relevant embedding has  $\Theta(g)$ faces. Hence every relevant embedding of $R_n$ has $\Theta(n+g)$ faces.
\smallskip

In order to prove Theorem~\ref{thm.conn} on connectivity, we also need to learn about the number $\ell(R_n)$  of leaves of~$R_n$. Roughly speaking, parts (a) and (b) of the next result say that when $g$ is small there will typically be linearly many leaves, and part (c) says that when $g$ is large there are likely to be few leaves.
Recall that $\rho(\cP)$ is the radius of convergence for the class $\cP$ of planar graphs, and $\rho(\cP) \approx 0.0367284$ by~\cite{Asymptoticformula}.
Let $S$ be any fixed surface and let $R^S_n \inu \cE^S$: then $\ell(R^S_n)$ is asymptotically normal with both mean and variance $\sim \rho(\cP) n$~\cite{Asymptoticformula,LimitLawsFixedS}; and \cc{by Corollary 1.3 in~\cite{PendAppComp}}, for any $\eps>0$
\begin{equation} \label{eqn.planarleaves}
\rho(\cP) -\eps < \ell(R^S_n)/n < \rho(\cP) +\eps \;\; \mbox{  wvhp} \, .
\end{equation}
We cannot be nearly as precise as this when $g(n)$ is not constant (though see Theorem~\ref{thm.heredleaves} (a) below).  Recall from \cite{MSsizes} that if $g(n)=O(n/\log n)$ then $\rho(\cA^g)>0$.

\begin{theorem} \label{thm.leaves}
\begin{description}
\item{(a)} If $g(n)$ is $o(n/\log^3 n)$ and $0<\alpha< \rho(\cP)\,$, then $\ell(R_n) > \alpha n \; \wvhp\,$.

\item{(b)} 
If $g(n)$ is $O(n/\log n)$ (so $\rho(\cA^g)>0$), $g$ is non-decreasing, and $0<\alpha< \rho(\cA^g)$, then
\[\limsup_{n \to \infty}\: \pr(\ell(R_n) > \alpha n) =1. \]

\item{(c)}  If $g(n)$ is $o(n^2)$ then $\ell(R_n) < \frac{3 n^2}{n+g}$ whp.
\m{\cc{extend (c) to larger $g$? - not now?}}
\end{description}
\end{theorem}
We will in fact also give a `lower density' version of Theorem~\ref{thm.leaves} (b), see Lemma~\ref{lemma.pvlowerdensity}, which is stronger in the sense that it asserts the existence of linearly many leaves whp for a subset of integers $n$ with lower density near 1, but weaker in that it only asserts the existence of a suitable constant $\alpha >0$. Our last two theorems in this section concern the maximum degree $\Delta(R_n)$ and the maximum face size.

\begin{theorem} \label{thm.Delta2}
\begin{description}
\item{(a)} There are constants $0 < c_1 < c_2$ such that, if $g(n)$ is $o(n/\log^3 n)$, then 
\[c_1 \log n \leqslant \Delta(R_n) \leqslant c_2 \log n \; \mbox{ whp}. \]
\item{(b)}
Let $g(n)$ be $O(n/\log n)$ and be non-decreasing, and let $0<\eps<1$. Then there are constants $0 < c_1 < c_2$ such that 
\[\limsup_{n \to \infty} \, \pr \left(c_1 \log n \leqslant \Delta(R_n) \leqslant c_2 \log n \right) \geqslant 1-\eps\,. \]
\end{description}
\end{theorem}
Much as with Theorem~\ref{thm.leaves} (b), we shall in fact prove a stronger `lower density' version of Theorem~\ref{thm.Delta2} (b), see Lemma~\ref{lemma.maxdeglowerdensity}. 
The next theorem concerns maximum face size.  We measure the size of a face by the length of the facial walk (that is, the number of times we traverse an edge).  Recall that we defined relevant embeddings just before Theorem~\ref{thm.faces}.
\begin{theorem} \label{thm.facesize}
\begin{description}
\item{(a)} There exist constants $0 < c_1 < c_2$ such that, if $g(n)$ is $o(n/\log^3 n)$, then whp, in every relevant embedding of $R_n$, the maximum face size is at least $c_1 \log n$ and at most $c_2 \log n$.
\item{(b)}
Let $g(n)$ be $O(n/\log n)$ and be non-decreasing, and let $0<\eps<1$. Then there are constants $0 < c_1 < c_2$ such that the following holds. For $n \in \N$ let $p_n$ be the probability that, in every embedding of $R_n$ in a surface of Euler genus at most $g(n)$, the maximum face size is at least $c_1 \log n$ and at most $c_2 \log n$.  Then $\limsup_{n \to \infty} p_n \geqslant 1-\eps$.
\m{check: can we get $\limsup = 1$?}
\end{description}
\end{theorem}
As with Theorem~\ref{thm.Delta2} (b), we shall in fact prove a stronger `lower density' version of Theorem~\ref{thm.facesize} (b), see Lemma~\ref{facesize_lowerdensity}.


\subsection{Random graphs from a hereditary embeddable class, $R_n \inu \hered(\cA^g)$}
\label{subsec.heredembed}

Given a class $\cB$ of graphs, we say that a graph $G$ is \emph{hereditarily} in $\cB$ if for each nonempty set $W$ of vertices the induced subgraph $G[W]$ is in  $\cB_{|W|}$; and we let $\hered(\cB)$ be the class of graphs which are hereditarily in $\cB$.
Given a genus function $g$ we are interested here in $\hered(\cA^g)$, the class of graphs which are hereditarily in $\cA^g$.
Since $\cP\subseteq \hered(\cA^g) \subseteq \cA^g$, we always have $\rho(\cP)\geqslant \rho(\hered(\cA^g))$, and Theorem~1(a) of \cite{MSsizes} (see Theorem~\ref{theorem:gc} (a) below) shows that $\hered(\cA^g)$ has growth constant $\gamma_{\cP}$ as long as $g(n) = o(n/\log^3\!n)$. We shall see that, if $g$ is sufficiently small then the fragment of a random graph $R_n$ in $\hered(\cA^g)$ looks like that from a random planar graph, whereas if $g$ is a little larger then whp $R_n$ is connected.
Recall that the fragment of a graph is the subgraph induced on the vertices not in the largest component.

In this section we present results on the fragment, the connectedness and the number $\ell(R_n)$ of 
leaves of $R_n\inu\hered(\cA^g)$, corresponding to Theorems~\ref{thm.conn} and~\ref{thm.leaves} in the embeddable case.  (We consider only these properties here.  See Lemma~\ref{lem.HFragplanar} for an analogue of Theorem~\ref{thm.Fragplanar}.)
Recall that, if $Y$ and $X_1,X_2,\ldots$ are random variables in some probability space, then $X_n$ \emph{converges in total variation} to $Y$ if for any $\eps>0$ there exists $N$ such that for every $n \geqslant N$, for every event $A$ we have $|\pr(X_n \in A) - \pr(Y \in A)|< \eps$.
This is a `uniform' form of convergence in distribution.
See Section~\ref{subsec.BPPRG} for the definition of the Boltzmann Poisson random planar graph $BP(\cP,\rho(\cP))$.  Recall that $C(x)$ is the exponential generating function of the class of connected planar graphs, and $p^*=e^{-C(\rho(\cP))} \approx 0.963$ is the limiting probability that a random planar graph is connected.

\begin{theorem} \label{thm.heredfrag}
Let the genus function $g(n)$ be $o(n/\log^3n)$,
and let $R_n \inu \hered(\cA^g)$.
Then $\Frag(R_n)$ converges in total variation
to $BP(\cP,\rho(\cP))$ as $n \to \infty$.  In particular, $\kappa(R_n)-1$ is distributed asymptotically as a Poisson law of parameter $C(\rho(\cP))$, and 
\[ \pr(R_n \mbox{ is connected}) \to p^* \mbox{ as } n \to \infty.\]
\end{theorem}
This theorem is much more precise than the corresponding result Theorem~\ref{thm.conn} (a) for the non-hereditary case.  Note in particular that it implies that $\Frag(R_n)$ is planar whp.
When $g$ is a little larger, connectivity undergoes 
a change in behaviour when $g(n)$ grows beyond about $n/\log n$.
\begin{theorem} \label{thm.heredconn}
Let $g$ be a genus function, and let $R_n \inu \hered(\cA^g)$. 
\begin{description}
\item{(a)}
If $g(n)$ is $O(n/\log n)$ and is non-decreasing then
\[\limsup_{n \to \infty} \, \pr(R_n \mbox{ is connected}) < 1\, . \]
\item{(b)}
If $g(n) \gg n/\log n$ then 
\[\limsup_{n \to \infty} \, \pr(R_n \mbox{ is connected}) =1. \]
\item{(c)}
If $g(n+1) \geqslant g(n) + 2$ for all $n \in \mathbb{N}$, then $R_n$ is connected whp.
\end{description}
\end{theorem}

Part (c) of the next result shows that there are likely be very few leaves when for example $g(n) = an$ with $a \geqslant 2$.  We obtain much tighter bounds than those in Theorem~\ref{thm.leaves} for the embeddable case.  Recall that for $x \in \R$ and $t \in \N$, $x_{(t)}$ denotes the product $x(x-1)\cdots(x-(t-1))$.
\begin{theorem} \label{thm.heredleaves}
Let $g$ be a genus function
and let $R_n \inu \hered(\cA^g)$. 
\begin{description}
\item{(a)}
If $g(n)$ is $o(n/\log^3 n)$ then for any $\eps>0$
\[ (1 - \eps)\, \rho(\cP)\, n < \ell(R_n) < (1 + \eps)\, \rho(\cP)\, n \;\;\; \wvhp.
\]
\item{(b)} Let $g(n)$ be $O(n/\log n)$ (so $\rho:=\rho(\hered(\cA^g))>0$), and let $g$ be non-decreasing.
Then for any $\eps>0$ 
\[\limsup_{n\rightarrow \infty} \, \mathbb{P}(\ell(R_n)> (1\!-\!\eps)\,\rho  n) = 1 \; \mbox{ and } \; \limsup_{n\rightarrow \infty} \, \mathbb{P}(\ell(R_n)< (1\!+\!\eps)\, \rho n) = 1\,.\]
\item{(c)} If $g(n+1) \geqslant g(n) + 2$ for all $n \in \mathbb{N}$, then $\E[\ell(R_n)]  \leqslant 2 +o(1)$; and indeed for each $t \in \N$ we have $\E[\ell(R_n)_{(t)}]  \leqslant 2^t +o(1)$.
\end{description}
\end{theorem}

\smallskip

We could be even more demanding and ask that each minor of our graphs (rather than each induced subgraph) is appropriately embeddable. It was shown in~\cite{MSsizes} that such minor-closed classes of embeddable graphs always either contain all graphs or have radius of convergence greater than zero. This leads to some drastic changes in typical properties around this threshold. See Section~\ref{minorclosed} where we consider random graphs from such minor-closed embeddable classes.

\subsection{Random unlabelled graphs $\tilde{R}_n$} \label{subsec.unlab}

We consider the probability of being connected for a random unlabelled graph.  Theorem~\ref{thm.unlab} concerns both an embeddable class and a hereditary embeddable class.  It gives a result for $\tilde{R}_n \inu \tA^g$ corresponding to Theorem~\ref{thm.conn} (b) in the labelled case; and for $\tilde{R}_n \inu \hered(\tA^g)$ corresponding to Theorem~\ref{thm.heredconn} (b) in the labelled case. 

\begin{theorem} \label{thm.unlab}
Let 
$g(n)$ be $O(n/\log n)\,$ and be non-decreasing; and either let
$\tilde{R}_n \inu \tilde{\cA}^g$ or let $\tilde{R}_n \inu \hered(\tA^g)$.
Then 
\[ \liminf_{n \to \infty}\, \pr(\tilde{R}_n \mbox{ is connected}) <1.\]
\end{theorem}
\noindent
In Section~\ref{sec.unlab} we prove the above result, and consider also random unlabelled graphs in the minor-closed case, see Theorems~\ref{thm.minorconnectedU} and~\ref{thm.unlabnew}.

\subsection{Plan of the paper}
\label{subsec.plan}
We have just presented our main results. The plan of the rest of the paper is as follows. In Section~\ref{sec.back} we give some background and preliminary results.  In Section~\ref{sec.edgesfaces} we prove Theorems~\ref{thm.edges} and \ref{thm.faces} on numbers of edges and faces, and Theorem~\ref{thm.leaves} on numbers of leaves. From Theorems~\ref{thm.edges} (b) and \ref{thm.faces} (b) we can deduce Corollary~\ref{cor.edgesfaces}, giving an estimate for certain large $g$ on the numbers of edges and faces whp in an embedding. In Section~\ref{section5} we prove Theorem~\ref{thm.leaves} on the number of leaves, and Theorems~\ref{thm.Delta2} and~\ref{thm.facesize} on the maximum degree and maximum face size of an embedding.
In Section~\ref{sec.frag} we prove Theorem~\ref{thm.Fragplanar} on the planarity of the fragment (the subgraph induced on the vertices not in the biggest component).
In Section~\ref{sec.conn} we complete the proofs of the results on random embeddable graphs presented in Subsection~\ref{subsec.embed} by proving Theorem~\ref{thm.conn} on connectedness.

In Section~\ref{sec.hered} we consider random hereditarily embeddable graphs and prove Theorems~\ref{thm.heredfrag}, \ref{thm.heredconn} and~\ref{thm.heredleaves}.  In Section~\ref{minorclosed} we consider minor-closed classes of graphs, and state and prove Theorems~\ref{thm.minorconnected},~\ref{thm.minorleaves},~\ref{thm.minormaxdeg}  and~\ref{thm.minorfacesize}.
In Section~\ref{sec.unlab} we 
discuss random unlabelled embeddable graphs, and in particular prove Theorem~\ref{thm.unlab}.
The final section, Section~\ref{sec.concl}, contains some concluding remarks and questions.

\bigskip


\section{Background and Preliminaries}
\label{sec.back}

In this section, we first introduce the Boltzmann Poisson random planar graph, then we collect some useful background theorems and observations on embeddings of graphs in surfaces, on growth ratios for graph classes, on structured classes of graphs (including embeddable classes), and on the size of an embeddable class $\cA^g$, 
which we will repeatedly make use of in the remainder of this paper.

\subsection{Boltzmann Poisson random planar graph}\label{subsec.BPPRG}

We present the special case for planar graphs of the Boltzmann Poisson random graph, see~\cite{randomGraphsMinorClosed}.
Let $P(x)$ be the exponential generating function for the class $\cP$ of (labelled) planar graphs, let $\rho:= \rho(\cP)$, and note that $P(\rho)$ is finite. Recall that  $\tP$ is the set of unlabelled planar graphs. Let
\[  \mu(H) = \frac{\rho^{v(H)}}{\aut(H)} \; \mbox{ for each } H \in \tP \]
with $\mu(\emptyset)=1$.  Here $v(H)$ is the number of vertices in $H$, and $\aut(H)$ is the size of the automorphism group, and $\emptyset$ is the empty graph, with no vertices.
Routine manipulations (see for example~\cite{randomGraphsMinorClosed}) show that
\[   P(\rho)= \sum_{H \in \tP} \mu(H). \]
The \emph{Boltzmann Poisson random planar graph} $R=BP(\cP,\rho)$ takes values in $\tP$, with
\[  \pr(R=H) = \frac{\mu(H)}{P(\rho)}  \;\; \mbox{ for each } H \in \tP.\]
Recall that $\cC$ is the class of connected graphs in $\cP$, with exponential generating function $C(x)$.
(By convention, $\emptyset$ is not in $\cC$.)  For each $H \in \tC$ let $\kappa(G,H)$ denote the number of components of $G$ isomorphic to $H$. Then (see \cite{randomGraphsMinorClosed}) the random variables $\kappa(R,H)$ for different $H \in \tC$ are independent, with $\kappa(R,H) \sim \Po(\mu(H))$. In particular, 
\[ \pr(R=\emptyset) = e^{-C(\rho)} = 1/P(\rho) = p^* \approx 0.963 \]
(see also the paragraph before Theorem~\ref{thm.conn}).
By Theorem 1.5 of \cite{randomGraphsMinorClosed} and Theorem 1.4 of~\cite{PendAppComp},
for $R_n \in \cP$ the fragment $\Frag(R_n)$ converges in total variation to $R = BP(\cP,\rho(\cP))$ as $n \to \infty$.
Thus $\kappa(\Frag(R_n))$ is distributed asymptotically as a Poisson law of parameter $\lambda= C(\rho)$; and in particular, the probability that $R_n$ is connected tends to $\pr(R=\emptyset)=p^*$.

\subsection{Embeddings of graphs in surfaces} \label{subsec.embeddings}

A central result which we will use repeatedly is Euler's formula. 
\medskip

\noindent
\emph{\bf Euler's formula}\;
Let the (finite) connected graph $G$, with $v$ vertices and $e$ edges, have a cellular embedding in a surface of Euler genus $h$, with $f$ faces. Then
\begin{equation} \label{eqn.Eulersformula}
v-e+f = 2-h \text{ .}
\end{equation}
\noindent
(This formula also holds for cellularly embedded connected pseudographs, that is graphs which may have multiple edges or loops.)

We need also to consider graphs which may not be connected.
Suppose that $G$ has $\kappa \geqslant 1$ components $H_1,\ldots,H_{\kappa}$.  If each component $H_i$ has a cellular embedding $\phi_i$ with $f_i$ faces and Euler genus $h_i$ then we say that $G$ has a cellular embedding $\phi$ with $f= \sum_i(f_i -1) +1 = \sum_i f_i -(\kappa-1)$ faces (we think of the `outer faces' of the $\kappa$ embeddings $\phi_i$ as being merged) and  Euler genus $h=\sum_i h_i$. The embedding $\phi$ is orientable if and only if each $\phi_i$ is orientable.
Extending (\ref{eqn.Eulersformula}), Euler's formula for graphs with $\kappa$ components is
\begin{equation} \label{eqn.Euler2}
v-e+f-\kappa  = 1-h \, .
\end{equation}

Now suppose that $G$ is a not necessarily connected simple graph with $v$ vertices and $e$ edges, which has a cellular embedding in a surface $S$ of Euler genus $h$, with $f$ faces. Then, since $G$ is simple, it must hold that $3f\leqslant 2e$. Thus by Euler's formula~(\ref{eqn.Euler2}) (and using $\kappa \geqslant 1$)
\begin{equation}\label{maxedges}
e \leqslant 3(v+h-2)
\end{equation}
and similarly
\begin{equation}\label{maxfaces}
f \leqslant 2(v+h-2).
\end{equation}
\noindent
Also by~(\ref{eqn.Euler2}), since $f \geqslant 1$ and $\kappa \leqslant v$, we have
\begin{equation}\label{eqn.egeqg}
e \geqslant h
\end{equation}
(see Theorem~\ref{thm.edges} (a)).

Finally here let us recall the Ringel-Youngs Theorem (see equation (7) in \cite{CompleteGraph}, or see for example Theorems 4.4.5 and 4.4.6 in~\cite{GraphsonSurfaces}). This says that the maximum Euler genus of a graph on $n$ vertices, that is the Euler genus of the complete graph $K_n$ on $n$ vertices, is equal to $2\lceil\tfrac1{12} (n-3)(n-4)\rceil \approx \tfrac16 n^2$ in the orientable case, and $\lceil \tfrac16 (n-3)(n-4)\rceil \approx \tfrac16 n^2$ in the non-orientable case, apart from $K_7$ where the non-orientable Euler genus is 3. 

\subsection{Growth ratios for graph classes}
In this section we present some bounds from~\cite{MSsizes} on the growth ratio when adding a vertex or adding two to the genus, 
which we will use repeatedly in the remainder of this paper. Recall that $\cA$ denotes any one of $\cE$ or $\cO\cE$ or $\cN\cE$ or $\cO\cE \cap \cN\cE$.
By~\cc{equation (24) of \cite{MSsizes}}
\begin{equation} \label{eqn.lbasymp}
\mbox{if } g(n) \ll n^2 \mbox{ and } n \mbox{ is sufficiently large, then} \;
\left|\cA_n^{h+2}\right| \geqslant \frac{n^2}{7(n+h)}\, \left|\cA_n^{h}\right| \: \mbox{ for all } h\in \mathbb{N}_0 \mbox{ with } h \leqslant g(n)\,.
\end{equation}
It follows that
\begin{equation} \label{eqn.notin}
 \mbox{ if the genus function } g 
 \mbox{ satisfies } g(n) \ll n^2 , \mbox{ and } R_n \inu \cA^g , \; \mbox{ then whp }\; R_n \not\in \cA^{g(n)-2}_n.   
\end{equation}
%
%
We have been considering changing the Euler genus by 2. Now we consider incrementing $n$ by 1.
By \cc{equation~(18) of \cite{MSsizes}} 
\begin{equation} \label{eqn.growth}
\mbox{for every } h \in \N_0 \mbox{ and } n \in \N \mbox{ we have } \;\;
|\cA^h_{n+1}| \geqslant 2n \, |\cA^h_{n}|.
\end{equation}

\subsection{Structured classes of graphs} \label{subsec.struc}
In this section, we collect some results about structured classes of graphs, including classes of graphs embeddable in surfaces, 
which have a growth constant or at least a positive radius of convergence (of the exponential generating function).
Probably the most important of these results, from which many other results follow, is the Pendant Appearances Theorem.  We need some definitions.

Let $G$ be a (large) graph, let $H$ be a (small) connected graph, and let $W$ be a proper subset of the vertex set $V(G)$.
We say that $G$ has a \emph{pendant appearance} of $H$ on $W$ if the induced subgraph $G[W]$ is isomorphic to $H$, and there is exactly one edge  in $G$ between $W$ and $V(G) \backslash W$ (the \emph{link edge}).  
Let $\pend(G,H)$ be the number of pendant appearances of $H$ in $G$ (that is, the number of sets $W \subseteq V(G)$ such that there is a pendant appearance of $H$ in $G$ on $W$).
We say that the connected graph $H$ can be \emph{attached to} a class $\cA$ of graphs if whenever we have a graph $G \in \cA$ and a disjoint copy of $H$ and we add an edge between a vertex in $G$ and a vertex in $H$ then the resulting graph must be in $\cA$. Similarly, we say that $H$ can be \emph{detached from} $\cA$ if whenever we have a graph $G \in \cA$ with a pendant appearance of $H$ on $W$ then $G \backslash W$ must be in $\cA$. Finally, we let $\aut(H)$ be the number of automorphisms of $H$.

We can now state the 
Pendant Appearances Theorem, as it is given in Corollary 1.2 in~\cite{PendAppComp}, improving on earlier versions of the theorem in~\cite{RandomPlanar,addable}.
\begin{theorem}\label{thm.app}~\cite{PendAppComp}
(\textbf{Pendant Appearances Theorem})
Let the class $\cB$ of graphs have radius of convergence $0< \rho(\cB) < \infty$, let $H$ be a connected graph, let $h=v(H)$ and let $\alpha_H = h\, \rho(\cB)^{h}/ \aut\, H$.
Let $0<\eps<1$.  Then there exists $\nu>0$ depending on $\rho(\cB), H$ and $\eps$ (but not on $\cB$ itself) such that the following hold.
\begin{description}
\item{(a)} If $H$ can be attached to $\cB$ then
\[ \rho(\{G \in \cB : \pend(G,H) \leqslant (1- \eps)\, \alpha_H \, v(G)\,\}) > \rho(\cB) + \nu\,. \]
\item{(b)} If $H$ can be detached from $\cB$ then
\[ \rho(\{G \in \cB : \pend(G,H) \geqslant (1+ \eps)\, \alpha_H \, v(G) \}) > \rho(\cB) +\nu \, .\] 
\end{description} 
\end{theorem}
In Theorem~\ref{thm.app}, suppose that $\cB$ has a growth constant, and let $R_n \inu \cB$.
Then by part (a) of the theorem, 
\begin{equation} \label{eqn.penda}
  \mbox{ if $H$ can be attached to $\cB$ \, then }\;\;
\pend(R_n, H) > (1- \eps)\, \alpha_H \, n \;\; \wvhp ;
\end{equation}
and by part (b),
\begin{equation} \label{eqn.pendb}
  \mbox{ if $H$ can be detached from $\cB$ \, then }\;\;
\pend(R_n, H) < (1+ \eps)\, \alpha_H \, n \;\; \wvhp\,.
\end{equation}
We will usually apply these results with $H$ as a single vertex, so we are talking about the number $\ell(R_n)$ of leaves in $R_n$. 
\smallskip

It may be natural to assume that the genus function $g$ is non-decreasing, and sometimes we will do so, but at other times this can be avoided.  For example, suppose that $H$ is a connected planar graph, and $g(n) \ll n/\log^3 n$, so $\cA^g$ has growth constant $\gamma_{\cP}$ by Theorem~\ref{theorem:gc}~(a).  If $g$ is non-decreasing, then $H$ can be attached to $\cA^g$, so we can apply equation~(\ref{eqn.penda}).
But we can avoid assuming that $g$ is non-decreasing, as follows.
Define the new genus function $g^*$ by setting $g^*(n) = \max \{ g(n') : 1 \leqslant n' \leqslant n\}$.  Then $g \leqslant g^*$ so $\cA^g \subseteq \cA^{g^*}$, $g^*$ is non-decreasing, and $g^*(n) \ll n/\log^3n$.
Thus 
$\cA^{g^*}$ has growth constant $\gamma_{\cP}$ as for $\cA^g$, so $|\cA^{g^*}_n| = (1+o(1))^n |\cA^g_n|$. Let $0<\eps< 1$.  Since $H$ can be attached to $\cA^{g^*}$,
by equation~(\ref{eqn.penda})  
there exists $\nu>0$ such that
\[ | \{ G \in \cA^{g^*}_n : \pend(G,H) < (1-\eps) \alpha_H n \}| \leqslant (1+o(1))\,e^{-\nu n} |\cA^{g^*}_n| = (1+o(1))^n e^{-\nu n} |\cA^g_n|\,.\]
Since $\cA^g \subseteq \cA^{g^*}$ this gives
\[ | \{ G \in \cA^{g}_n : \pend(G,H) < (1-\eps) \alpha_H n \}| \leqslant  (1+o(1))^n e^{-\nu n} |\cA^g_n|\,,\]
so that
\begin{equation} \label{eqn.pend-g}
\mbox{ for } R_n \inu \cA^g\,, \;\;\; \pend(R_n,H) \geqslant (1-\eps) \alpha_H n \;\; \wvhp
\end{equation}
(without assuming that $g$ is non-decreasing).
\smallskip


%
%


A set $\cA$ of graphs is called \emph{bridge-addable} when for each graph $G$ in $\cA$, if $u$ and $v$ are vertices in distinct components of $G$ then the graph $G+uv$ obtained from $G$ by adding an edge between $u$ and $v$ is also in $\cA$. The concept of being bridge-addable was introduced in~\cite{RandomPlanar} (though initially called weakly addable). If a set $\cA$ of graphs is bridge-addable then we can give upper bounds on 
the number of components and the size of the fragment of a graph $R_n \inu \cA$ chosen uniformly at random from $\cA_n$. Recall that $\frag(G)$ is the number of vertices in the fragment of the graph $G$, and $\Po(\lambda)$ denotes the Poisson distribution with mean $\lambda$.
\begin{lemma}\label{lem.frag} \cite[Theorem 2.2]{addable} \cite[Equation (7)]{C-wba}
If $\cA$ is a bridge-addable set of graphs and $R_n \inu \cA$, then (a) $\kappa(R_n)$ is stochastically at most $1+\Po(1)$, so $\pr(R_n \mbox{ is connected}) \geqslant 1/e$ and $\, \E[\kappa(R_n)] <2$ ; and (b) $\, \E[\frag(R_n)] <2$.
\end{lemma}


\subsection{Size of an embeddable class $\cA^g$ of graphs} \label{subsec.size}

The following theorem is one of the main results of the companion paper~\cite{MSsizes}. It is stated here for convenience, since we will make repeated use of it.
\begin{theorem}\cc{\cite[Theorem 1]{MSsizes}}\label{theorem:gc}
Let $g=g(n)$ be a genus function; and let $\cA^g$ denote one of the four classes $\cE^g$, $\cO\cE^g$, $\cN\cE^g$ or $\cO\cE^g \cap \cN\cE^g$.

(a) If $g(n)$ is $o(n / \log^3n)$ then $\cA^{g}$ has growth constant $\gamma_{\mathcal{P}}$; that is,
\begin{equation}
\left|\cA_n^{g}\right| = (1+o(1))^n \, \gamma_{\cP}^n \; n! \, .\notag
\end{equation}

(b)
If $g(n)$ is $O(n)$ then
\begin{equation}
\left|\cA_n^{g}\right| = 2^{\Theta(n)} \, g^{g} \; n! \;\;\; \mbox{ and } \;\;\; |\tA_n^{g}| = 2^{\Theta(n)} \, g^{g} \, .\notag
\end{equation}
\end{theorem}
\noindent
Observe that by part (b)
\begin{equation} \label{eqn.rhopos}
\rho(\cA^g) \geqslant \tilde{\rho}(\tA^g) >0 \;\;\; \mbox{ if }\; g(n)=O(n/\log n).
\end{equation}


\section{Numbers of edges and faces of $R_n\inu\cA^g$} \label{sec.edgesfaces}
In this section we prove Theorem~\ref{thm.edges} on the number $e(R_n)$ of edges of $R_n$, and Theorem~\ref{thm.faces} concerning the numbers of faces in embeddings of $R_n$. We first prove Theorem~\ref{thm.faces}~(a). Then in the next subsection we prove Theorem~\ref{thm.edges}, together with a theorem (Theorem~\ref{thm.expectededges}, not presented earlier) about the expected number of edges in $R_n$. We make use of the result of Theorem~\ref{thm.faces}~(a) in the proof of Theorem~\ref{thm.edges}~(a). Finally, we prove Theorem~\ref{thm.faces}~(b) by making use of the result of Theorem~\ref{thm.edges}~(b). 

\subsection{Numbers of faces: proof of Theorem \ref{thm.faces} (a)}
\label{subsec.facesproof(a)}
In this section, we use a double counting argument to prove  Theorem \ref{thm.faces}~(a), which shows that whp every relevant embedding of a random graph $R_n \in_u \cA^g$ has many faces.
\begin{proof}[Proof of Theorem~\ref{thm.faces} (a)]
Let $n \in \N$.  For each $f \in \N$ let
\[\cB_n(f) = \left\{G\in \cA^g_n \mid G \text{ has a relevant embedding with at most } f \text{ faces}\right\}  .\]
Let $c_0=1/14$, let $f_0= c_0n^2/(n+g)$, let $f_1 \leqslant f_0$, and let $G\in \cB_n(f_1)$ have a relevant embedding with $f \leqslant f_1$ faces. We allocate each vertex of $G$ to one of its incident faces. By adding an edge between two vertices that have been allocated to the same face (as long as the edge is not already present in $G$), we create a graph $G' \in \cB_n(f_1+1)$ embedded in the same surface with $f+1$ faces. Note that we avoid counting any potential edges twice by assigning each vertex to a unique face. Let $n_i$ be the number of vertices we have allocated to the $i$th face. Then the number of possible edges between vertices that have been allocated to the same face is
\begin{equation}
    \begin{split}
        \sum_{i=1}^{f} \binom{n_i}{2}
        &=\left(\sum_{i=1}^{f}\frac{n_i^2}{2}\right) - \left(\sum_{i=1}^{f} \frac{n_i}{2}\right)\notag\\
        &\geqslant \left(\frac{1}{2}\cdot \sum_{i=1}^{f}\left(\frac{n}{f}\right)^2\right) -\frac{n}{2}\notag\\
        &=\frac{n^2}{2f} - \frac{n}{2} \; \geqslant \; \frac{n+g}{2c_0} - \frac{n}2.
    \end{split}
\end{equation}
Further, there are at most $3(n+g)$ edges already present in the graph $G$ by equation (\ref{maxedges}). Hence, from $G$ we construct at least
\[\frac{n+g}{2c_0} - \frac{n}{2} - 3(n+g) = 4(n+g) - \frac{n}{2} \geqslant \tfrac72 (n+g)\]
graphs $G' \in \cB_n(f_1+1)$.
Since there are at most $3(n+g)$ edges in $G'$, each graph $G'$ is constructed at most $3(n+g)$ times. Thus
\begin{equation}
\left|\cB_n(f_1)\right| 
\leqslant \tfrac67 \, \left|\cB_n(f_1+1)\right| \text{ .}
\end{equation}
Then for each $i \in \N$, $|\cB_n(f_0-i)| \leqslant (\frac67)^i \, |\cB_n(f_0)|$.  Let $f_2 = c\, n^2/(n+g)$ where $c = 1/15$.  Then $f_0-f_2 \to \infty$ as $n \to \infty$, so
\[ |\cB_n(f_2)| \leqslant (\tfrac67)^{f_0-f_2-1} |\cB_n(f_0)| \ll |\cB_n(f_0)| \leqslant |\cA^g_n|.\]
Hence, whp every relevant embedding of $R_n$ has more than $f_2$ faces, which completes the proof.
\end{proof}

\subsection{Numbers of edges: proof of Theorem~\ref{thm.edges}}
\label{subsec.edgesproof}

In this subsection we prove Theorem~\ref{thm.edges} on the number $e(R_n)$ of edges. To prove part (a) of Theorem~\ref{thm.edges} we make use of Theorem~\ref{thm.faces} (a) which was proven in Section~\ref{subsec.facesproof(a)}.
We also state and prove Theorem~\ref{thm.expectededges}, which gives another whp lower bound on $e(R_n)$.
\smallskip

\begin{proof}[Proof of Theorem~\ref{thm.edges} (a)] 
By equation~(\ref{eqn.notin}), whp $R_n \not\in \cA^{g(n)-2}$, and so $R_n$ has a cellular embedding $\phi$ in a surface of Euler genus $\geqslant g(n)-1$.
By Euler's formula, equation~(\ref{eqn.Euler2}), it follows that whp
\[ n-e(R_n) +f -\kappa(R_n) \leqslant 1-(g-1) = 2-g,\]
where $\kappa(R_n)$ is the number of components of $R_n$ and $f$ is the number of faces in the embedding $\phi$;
and so
\[ e(R_n) \geqslant n+g + f - \kappa(R_n) -2. \]
By part (a) of Theorem~\ref{thm.faces}, whp $f \geqslant \tfrac1{15} \frac{n^2}{n+g(n)}$, and this bound tends to infinity as $n \to \infty$.
But since $\cA^g_n$ is bridge-addable, by Lemma~\ref{lem.frag} we have $\mathbb{E}[\kappa(R_n)]<2$. Thus by Markov's inequality, whp $f-\kappa(R_n) -2 > 0$. Hence
$e(R_n) > n+g(n)$ whp, as required.  
\end{proof}

There are results that give bounds on the genus of most graphs with a given number of edges \cite{GenusRandom1, GenusRandom3, GenusRandom2}. Some of these results are stated for the binomial 
random graph $G(n,p)$ with given edge probability $p=p(n)$, but they apply to the case of a given number $m=m(n)$ of edges, as pointed out in \cite{GenusRandom1, GenusRandom3}. We make use of the upper bounds in the proof of Theorem~\ref{thm.edges} (b).
\begin{proof}[Proof of Theorem~\ref{thm.edges} (b)] 
Let $j\in \mathbb{N}$ and let the genus function $g$ satisfy $n^{1+1/(j+1)} \ll g(n) \ll n^{1+1/j}$. Further, let $\eps >0$ and $m = m(n)=\lfloor(1-\tfrac12 \eps)\, \tfrac{j+2}{j}\, g(n) \rfloor$. It follows from \cite{GenusRandom2} (see (1.2) for the $G(n,p)$ version) that almost every graph on $n$ vertices with $m$ edges can be embedded in an orientable surface of Euler genus at most
\[ (1+ \tfrac12 \eps)\, \tfrac12 \, \tfrac{j}{j+2} \left(m/\tbinom{n}{2}\right)\, n^2 \leqslant (1- \tfrac14 \eps^2)\, \tfrac{n}{n-1}\, g(n) \leqslant g(n) -1\]
for $n$ sufficiently large.
From Observation \cc{10 in \cite{MSsizes}} it then follows that almost every graph on $n$ vertices with $m$ edges can be embedded in a non-orientable surface of Euler genus at most $g(n)$. So, for $n$ sufficiently large, at least half of the graphs on $n$ vertices with $m$ edges lie in the class $\cA_n^g$.

Now consider further the case $j=1$.  We want to show that the last statement still holds when we extend the range of values $g(n)$ up to $\tfrac1{12} n^2$.  By the above argument, we may suppose that $g(n)$ is large, say $g(n) \geqslant n^2/\log n\,$; and $g(n) \leqslant \tfrac1{12} n^2$.
Then $p=m/\binom{n}{2}$ satisfies $p^2(1-p^2) = \Omega((\log n)^{-2}) \gg 8(\log n)^4/n$.
Hence, it follows from \cite[Theorem 4.5]{GenusRandom1} that almost every graph on $n$ vertices with $m$ edges embeds in an orientable and a non-orientable surface of Euler genus at most $(1+\tfrac12 \eps)\, \tfrac13 m \leqslant g(n)$. So, for $n$ sufficiently large, at least half of the graphs on $n$ vertices and with $m$ edges lie in the class $\cA_n^g$, as desired.

Let $m^- = m^-(n) = \lfloor(1- \eps)\, \tfrac{j+2}{j}\, g(n) \rfloor$.  Then the number of graphs on $[n]$ with at most $m^-$ edges is
\[ \sum_{k=0}^{m^-}\binom{\binom{n}{2}}{k} \leqslant
(m^- +1)\, \binom{\binom{n}{2}}{m^-}
 \ll \tfrac{1}{2}\binom{\binom{n}{2}}{m} \leqslant |\cA^g_n|
\]
for $n$ sufficiently large.
Hence whp $e(R_n) > m^-$, and the required result follows.
\end{proof}
\smallskip

\noindent\emph{Another lower bound on numbers of edges}

Another way of bounding the number $e(R_n)$ of edges is by considering the number of edges an edge-maximal graph can have.
One of the earliest questions investigated concerning the random planar graph $R_n \inu \cP$ was about $e(R_n)$; and Theorem~1 of Denise et al~\cite{TherandomPlanar} 
in 1996 says that
$ \E[e(R_n)] \, \geqslant \, \tfrac12 (3n-6)$.  The proof was based on the fact that an $n$-vertex plane triangulation has $3n-6$ edges.
(See also~\cite{EdgesRand,RandTriang,PlanGraphsViaWell}.)
In fact a stronger and more general result holds.
Let $\Bin(k,p)$ denote the binomial distribution, for the number of successes in $k$ independent trials each with probability $p$ of success.
\begin{theorem}\label{thm.expectededges}
Let $g$ be any genus function, let $\cA^g$ be $\cE^g$ or $\cO\cE^g$ or $\cN\cE^g$, and let $R_n\in \cA^g_n$. Then for each $n \in \N$ we have $e(R_n) \geqslant_s \Bin(3n-6,\frac12)$; that is, $e(R_n)$ is stochastically at least $\,\Bin(3n-6,\frac12)$.
\end{theorem}
\noindent
Theorem~\ref{thm.expectededges} implies in particular that $\E[e(R_n)] \geqslant \frac{1}{2} (3n-6)$.
\begin{proof}
By~\cite{DaviesThomassen2020}, for every surface $S$ and every $n \in \N$, every edge-maximal graph in $\cE^S_n$ has at least $3n-6$ edges.  (This improves on a result from~\cite{edgemaxgraphsonsurfaces}.)
Thus every edge-maximal graph in $\cA^g_n$ has at least $3n-6$ edges, and so the theorem follows immediately from the next general lemma.
\end{proof}

\begin{lemma} \label{Clem.dvw1}
Let $E$ be a finite set, and let $\cE$ be a non-empty collection of subsets of $E$ which is closed downwards (that is, if $A \in \cE$ and $B \subset A$ then $B \in \cE$).
Suppose that all maximal members of $\cE$ have size at least $r$, for some $r \in \N$. 
Then for $R \inu \cE$ we have
$|R| \geqslant_s \Bin(r,\frac12)$.
\end{lemma}
It remains to prove Lemma~\ref{Clem.dvw1}: we use one preliminary lemma.  

\begin{lemma} \label{Clem.dvw2}
Let $M$ be a set of size $m \geqslant 1$, and let $\cA$ be a non-empty collection of subsets of $M$ which is closed upwards (that is, if $A \in \cA$ and $A \subseteq B \subseteq M$ then $B \in \cA$). 
Then for $R \inu \cA$ we have $|R| \geqslant_s \Bin(m,\frac12)$.
\end{lemma}

\begin{proof}[Proof of Lemma~\ref{Clem.dvw2}]
Let $S \inu 2^M$ (so $|S| \sim \Bin(m,\frac12)$), and let $f: \R \to \R$ be a non-decreasing function. It suffices to show that $\E[f(|R|)] \geqslant \E[f(|S|)]$.  For $A \subseteq M$ let $\bar{f}(A)= f(|A|)$ and let $h(A)= \ind_{A \in \cA}$. Then $\bar{f}$ and $h$ are non-decreasing functions on the lattice of subsets of $M$.  Hence by Harris's inequality~\cite{HarrisIneq} (or see for example~\cite{Grimmett2018})
\begin{eqnarray*}
  2^{-m} \sum_{A \in \cA} f(|A|)
&=&
  \sum_{A \subseteq M} 2^{-m} \bar{f}(A)h(A)\\
& \geqslant & 
  \sum_{A \subseteq M} 2^{-m} \bar{f}(A)\cdot \sum_{A \subseteq M} 2^{-m} h(A)\\
&=&
  \E[f(|S|)] \cdot 2^{-m} |\cA|.
\end{eqnarray*}
Thus
\[ \E[f(|R|)] = |\cA|^{-1} \sum_{A \in \cA} f(|A|) \geqslant \E[f(|S|)] \]
as required.
\end{proof}
\begin{proof}[Proof of Lemma~\ref{Clem.dvw1}]
List the maximal sets in $\cE$ as $M_1,\ldots,M_k$ for some $k \geqslant 1$.  For each set $A \in \cE$ let $\phi(A)$ be the least $i$ such that $A \subseteq M_i$.  For $i=1,\ldots,k$ let $\cE_i$ be the collection of sets $A \in \cE$ such that 
$\phi(A)=i$, and note that
$\cE_i$ is a collection of subsets of $M_i$ which is closed upwards in $M_i$ (and contains $M_i$).  Let $Y \sim \Bin(r,\frac12)$.  Then by Lemma~\ref{Clem.dvw2}, for each $t \geqslant 0$
\begin{eqnarray*}
  \pr(|R| \geqslant t)
& = &
  \sum_{i=1}^k \frac{|\cE_i|}{|\cE|}\: \pr(|R| \geqslant t \,|\, R \in \cE_i)\\
& \geqslant &
  \sum_{i=1}^k \frac{|\cE_i|}{|\cE|} \: \pr(Y \geqslant t) \; = \; \pr(Y \geqslant t). 
\end{eqnarray*}
Thus $|R| \geqslant_s Y$, as required.
\end{proof}

\subsection{Numbers of faces: proof of Theorem~\ref{thm.faces}  (b)}
Finally, we prove part (b) of Theorem~\ref{thm.faces}. We shall use
Theorem~\ref{thm.edges} (b), which was proven in Section~\ref{subsec.edgesproof}.  We shall also use Theorem~\ref{thm.conn} part (c), which has not yet been proved, but its proof is independent of Theorem~\ref{thm.faces}.
\begin{proof}[Proof of Theorem \ref{thm.faces} (b)] 
Let $j\in \mathbb{N}$, and let the genus function $g$ satisfy $n^{1+1/(j+1)}\ll g(n)\ll n^{1+1/j}$, except that when $j=1$ we raise the upper bound to $g(n) \leqslant \frac1{12} n^2$. Let $\eps > 0$, and let $R_n \inu \cA_n^g$.
By Theorem~\ref{thm.edges} part (b) (with $\eps$ replaced by $\tfrac{\eps}{2j}$),  whp $e(R_n) \geqslant (1- \tfrac{\eps}{2j})\, \tfrac{j+2}{j} \, g$.  (Here as earlier we write $g$ for $g(n)$.)
Also, by Theorem~\ref{thm.conn} part (c) whp $R_n$ is connected.  To prove the result it is thus sufficient to show that, if $G$ is a connected $n$-vertex graph with $e(G) \geqslant (1- \tfrac{\eps}{2j})\,\tfrac{j+2}{j} \, g$, then any cellular embedding $\phi$ of $G$ in a surface of Euler genus at most $g$ has at least $(1-\eps)\,\tfrac{2}{j}\, g$ faces, if $n$ is sufficiently large.
By Euler's formula, see equation~(\ref{eqn.Eulersformula}), the number $f$ of faces in the embedding $\phi$ satisfies
\begin{equation}
\begin{split}
    f & \geqslant e(G)-g-n+2 \;\;
    \geqslant (1- \tfrac{\eps}{2j})\,\tfrac{j+2}{j} \, g\, -g-n\\
    &= (\tfrac{2}{j} - \tfrac{\eps}{2j} - \tfrac{\eps}{j^2})\, g -n \;\;
    \geqslant (1-\eps)\, \tfrac{2}{j}\,g
\end{split}
\end{equation}
for $n$ sufficiently large, as required.
\end{proof}

We have shown that when $g(n) = o(n^2)$, the number of edges in $R_n$ is whp at least $n+g$. Similarly, for each embedding the number of faces is whp at least $n^2/15(n+g)$. 
For larger genus, we have improved on these lower bounds. Once $g(n) \gg n^{3/2}$ whp there are $(3+o(1))g$ edges in the graph and each embedding has $(2+o(1))g$ faces. This implies in particular, that each embedding is with high probability close to being a triangulation. For smaller genus, the best currently known upper bound on the number of edges is $3(n+g-2)$ from Euler's formula~(\ref{eqn.Eulersformula}), which is much larger than the lower bound of $n+g$. Can we give more precise bounds on the number of edges in this case?

For any given surface $S$, is a graph $R^S_n$ chosen uniformly at random from all graphs on $n$ vertices embeddable in the surface $S$ likely to have at least as many edges as a random planar graph on $n$ vertices?
We make a conjecture.
\begin{conjecture}
For any surface $S$ and $n \in \N$, we have
$\,\E[e(R^S_n)] \geqslant \E[e(R^{\mathbf{S}_0}_n)]$, where $\mathbf{S}_0$ denotes the sphere.
\end{conjecture}
This conjecture can be extended by asking whether $e(R^S_n) \geqslant_s e(R^{\mathbf{S}_0}_n)$ (that is,  $e(R^S_n)$ stochastically dominates $e(R^{\mathbf{S}_0}_n)$).
Another possible extension of the conjecture would be: let $S$ be any surface and let $S^+$ be the surface obtained by adding to $S$ a handle or cross-cap. Then for each $n \in \N$ can we couple $R^S_n \inu \cE^S$ and $R^{S+}_n \inu \cE^{S^+}$ so that $E(R^S_n) \subseteq E(R^{S+}_n)$?
Clearly, similar questions can be asked for the number of faces in an embedding of $R_n$.

\section{Number of leaves, maximum degree and face size of $R_n\inu\cA^g$}\label{section5}
In this section, we first prove Theorem~\ref{thm.leaves} on the number $\ell(R_n)$ of leaves in a random graph $R_n \inu \cA^g$, and present a corresponding lower density version (from \cc{\cite{MSsizes}}) in Lemma~\ref{lemma.pvlowerdensity}. We then go on to prove Theorem~\ref{thm.Delta2} and~\ref{thm.facesize} which give lower and upper bounds on the maximum degree $\Delta(R_n)$ and the maximum face size of a relevant embedding of $R_n$ respectively. We give lower density versions of these theorems in Lemmas~\ref{lemma.maxdeglowerdensity} and~\ref{facesize_lowerdensity}.

\subsection{Numbers of leaves: proof of Theorem~\ref{thm.leaves}}
\label{sec.leaves}
In this section we prove Theorem~\ref{thm.leaves} on the number $\ell(R_n)$ of leaves, and present a corresponding lower density result in Lemma~\ref{lemma.pvlowerdensity}.

\begin{proof}[Proof of Theorem~\ref{thm.leaves} (a)]
This follows directly from~(\ref{eqn.pend-g}).
%
\end{proof}

\begin{proof}[Proof of Theorem~\ref{thm.leaves} (b)]
Suppose that $g(n)=O(n/\log n)$ and $g$ is non-decreasing. By Theorem \ref{theorem:gc} (b), $\cA^g$ has radius of convergence $\rho(\cA^g)>0$.  Also, single vertices are attachable to $\cA^g$ (to form leaves) since $g$ is non-decreasing.
Let $0<\alpha<\rho(\cA^g)$. 
Let $a_n = (|\cA^g_n|/n!)^{1/n}$ for each $n \in \N$.  Then $\limsup_{n \to \infty} a_n = \rho(\cA^g)^{-1}$.
Let $n_1 < n_2< \cdots$ be such that $a_{n_i} \to \rho(\cA^g)^{-1}$ as $i \to \infty$. Then $\pr(\ell(R_{n_i}) <  \alpha n_i) = e^{-\Omega(n_i)} = o(1)$ as $i \to \infty$ by Theorem~\ref{thm.app} (see also equation~(\ref{eqn.penda})). 
\end{proof}

\begin{proof}[Proof of Theorem~\ref{thm.leaves} (c)]
Suppose that $g(n) \ll n^2$.  We must show that $\ell(R_n) < \frac{3n^2}{n+g}$ whp.
Since $R_n$ can have at most $n$ leaves, we may assume that $g(n) \geqslant 2n$ for each $n \in \N$; and it suffices to show that $\ell(R_n) < 2 n^2/g$ whp. For each $n \in \N$ let \[ \cB_n = \{G \in \cA_n^{g} \mid \ell(G) \geqslant 2 n^2 / g \text{ and } e(G) \geqslant n+g \}. \]
Since $e(R_n) \geqslant n+g$ whp by Theorem~\ref{thm.edges}, it suffices to show that whp $R_n \not\in \cB_n$, that is $|\cB_n|/|\cA^g_n| = o(1)$. The idea of the proof is to show that from the graphs in $\cB_n$ we can construct many graphs in $\cA^g_n$ with little double counting, so we cannot have started with many graphs in~$\cB_n$.

Let $n \in \N$, assume that $\cB_n \neq \emptyset$, and let $G\in\cB_n$. Let $k= \lceil n^2 / g \rceil$, and choose a set $\{v_1,\ldots,v_k\}$ out of the at least $2k-1$ leaves of $G$. 
Delete the $k$ edges incident to the $k$ chosen vertices $v_i$ to form $G^-$, and then insert the vertices $v_i$ one at a time in the middle of an edge. Since $e(G) \geqslant n+ g$ and $k \leqslant n$ we have $e(G^-) \geqslant n + g -k \geqslant g$; and thus we have at least
$\tbinom{2k-1}{k}\cdot g^{k}$
choices of where to insert the vertices $v_i$. (Note that several vertices $v_i$ can be inserted in the same edge from $G^-$.) Let the resulting graph be $G'$, and note that $G' \in \cA^g_n$.

How often is each such graph $G'$ constructed? First we guess the set of $k$ vertices $v_1,...,v_k$: there are at most $\binom{n}{k}$ choices. Now we guess the $k$ vertices to which the vertices $v_1,...,v_k$ were originally attached: there are at most $n^{k}$ choices for this. So every graph $G'$ is constructed at most
\[\tbinom{n}{k} \cdot n^{k} \leqslant \big(\frac{e \, n^2}{k}\big)^{k} \leqslant \left(e \, g\right)^{k}\]
times. It follows that
\[ \left|\cB_n\right|  \cdot \tbinom{2k-1}{k} g^k \leqslant \left|\cA^g_n\right| \cdot (e \, g)^k. \]
But $\binom{2k-1}{k} = (4+o(1))^k$ as $k \to \infty$, so
\[ \frac{\left|\cB_n\right|}{\left|\cA^g_n\right|} \leqslant  e^{k} \cdot \tbinom{2k-1}{k}^{-1} = (e/4 +o(1))^k \;\; \mbox{ as } n \to \infty.\]
Thus $|\cB_n|/|\cA^g_n| = o(1)$ as $n \to \infty$, as required.
\end{proof}

As long as the genus is not too big, there is also a lower-density result on the number of leaves, and indeed such a result was needed in~\cite{MSsizes}. Given $0<\delta <1$ we say that a set $I \subseteq {\mathbb N}$ has \emph{(asymptotic) lower density at least} $\delta$ if 
$|I \cap [n]| \geqslant \delta n$ for all sufficiently large $n \in {\mathbb N}$.
Let 
$g(n)$ be $O(n/\log n)$, and let $0< \eps <1$.  Then, by \cc{Lemma 28 of \cite{MSsizes}}, there exists a constant $c=c(g,\eps)$ such that the set $I^*(g,\eps)$
of integers $n \geqslant 1$ for which
\begin{equation} \label{eqn.slow}
\left|\cA_{n+1}^{g}\right| \leqslant c\, (n+1)  \left|\cA_{n}^{g}\right|
\end{equation}
has lower density at least $1-\eps$.

\begin{lemma}\cc{\cite[Lemma 29]{MSsizes}}\label{lemma.pvlowerdensity}
Let 
$g(n)$ be $O(n/\log n)$ and be non-decreasing, and let $0< \eps, \, p <1$.
Let $R_n \inu \cA^{g}$.
Then there exist $\alpha>0$ and $n_0\in \mathbb{N}$ such that for all $n\geqslant n_0$ with 
$n \in I^*(g,\eps)$ 
\begin{equation} \notag
    \mathbb{P}(R_n \text{ has at least } \alpha n \text{ leaves}) \geqslant p \text{ .}
\end{equation}
\end{lemma}
This result is used in the proof of Lemma~\ref{lemma.connectedLD}, from which we deduce Theorem~\ref{thm.conn} (b).
Note that it is stronger than Theorem~\ref{thm.leaves} (b) in that it concerns a subset of integers $n$ with lower density near 1, but weaker in that it only asserts the existence of a suitable constant $\alpha >0$.

Currently we know from Theorem~\ref{thm.leaves} that in the embeddable case, when $g(n)$ is $o(n/\log^3 n)$ there is a linear number of leaves whp; 
when $g(n)$ is $O(n/\log n)$ the limsup of the probability that there are linearly many leaves is equal to 1; and 
when $g(n)$ is $o(n^2)$ the number of leaves is bounded above by $3 n^2/(n+g)$ whp,
which implies that as soon as $g(n)\gg n$ the graph $R_n$ has whp 
a sublinear number of leaves. Can we find corresponding lower bounds for these results and can we find bounds for 
the range $n/\log n \ll g(n) \ll n$? Further, is there a sharp phase transition for the number of leaves occurring at around $n/\log n$, at around $n$, or somewhere in between; or is the transition actually not sharp and the typical number of leaves slowly decreases when $g(n)$ increases from $n/\log n$ towards $n^2$, so does the number of leaves correspond more to our current upper bound?

In the hereditary case, see Section~\ref{sec.hered}, we can be more precise and in particular when $g(n)$ satisfies $g(n+1)\geqslant g(n)+2$ for all $n$, the graph $R_n$ has at most a constant number of leaves whp. 
Does a similar result hold in the embeddable case?

\subsection{Maximum degree: proof of upper bound in Theorem~\ref{thm.Delta2} (a)} \label{upperboundmaxdeg}
For the proof of the upper bound in Theorem~\ref{thm.Delta2} (a) on the maximum degree $\Delta(R_n)$, it is convenient to state two preliminary lemmas from \cite{MSsizes}, Lemmas~\ref{lem.fewpend} and~\ref{lem.Delta} here. We also use our earlier results on numbers of leaves. We further state and prove a corresponding lower density version for the case when $g(n)$ is $O(n/\log n)$ in Lemma~\ref{lemma.maxdegLDupperbound}, which will be used for the proof of Theorem~\ref{thm.Delta2} (b) in Section~\ref{degfacesize_LDversions}. 
Lemma~\ref{lem.fewpend} concerns the maximum number of leaves adjacent to any vertex.
\begin{lemma}\cc{\cite[Lemma 30]{MSsizes}} \label{lem.fewpend}
Let $\cG$ be a class of graphs which is closed under (simultaneously) detaching and re-attaching any leaf, and let $R_n \inu \cG$. Then whp each vertex in $R_n$ is adjacent to at most $2 \log n /\log\log n$ leaves.
\end{lemma}

Let $\cS$ be the set of graphs $G$ such that if $G$ has $n$ vertices then each vertex is adjacent to at most $2\log n/\log\log n$ leaves (where $\cS$ is for {\bf s}mall).  Since $\cA_n^g$ is closed under (simultaneously) detaching and re-attaching leaves, by Lemma~\ref{lem.fewpend} we have $R_n \in \cS$ whp. Now, given $0<\alpha<1$, let $\cL^\alpha$ be the set of graphs $G$ such that if $G$ has $n$ vertices then it has at least $\alpha n$ leaves.  The next lemma concerns $\cS$, $\cL^\alpha$ and maximum degree.
\begin{lemma}\cc{\cite[Lemma 31]{MSsizes}} \label{lem.Delta}
Let $0<\alpha<1$, let $b=b(n) = \lceil (8/\alpha)\, \log n \rceil$, and let
\[\cB = \{ G \in \cL^\alpha  \cap \cS : \Delta(G) \geqslant b(n) \mbox{ where } n=v(G) \}. \]
There is a function $\eta(n)=o(1)$ as $n \to \infty$ such that the following holds: for all $n \in \N$ and all surfaces $S$, the random graph $R^S_n \inu \cE^S$
satisfies $\pr(R^S_n \in \cB) \leqslant \eta(n)$.
\end{lemma}

We can now obtain the desired upper bounds on the maximum degree in Theorem~\ref{thm.Delta2} (a).

\begin{proof}[Proof of Theorem~\ref{thm.Delta2} (a) upper bound]
Let $g(n)$ be $o(n/\log^3n)$.  By Theorem \ref{theorem:gc} (a), $\cA^g$ has growth constant~$\gamma_\cP$. 
Let $0< \alpha <\gamma_\cP^{-1}$. Then by equation~(\ref{eqn.penda}), 
whp $R_n$ has at least $\alpha n$ leaves, that is $R_n \in \cL^\alpha$.  We have already seen that whp $R_n \in \cS$.  Hence by Lemma~\ref{lem.Delta}, whp $\Delta(R_n)<b(n)$, where $b(n) = \lceil (8/\alpha) \log n \rceil$.
\end{proof}
We now give a lower density result for the maximum degree, for which our proof requires $g$ to be non-decreasing. Note that this lemma was proven in a slightly weaker form in \cite{MSsizes}.
\begin{lemma}\label{lemma.maxdegLDupperbound}
Let $g(n)$ be $O(n/\log n)$ 
and be non-decreasing, 
and let $R_n\in_u \cA^g$. For any $\eps>0$ there exists a constant $c>0$ such that the following holds. For $n\in\mathbb{N}$ let
\[p_n = \mathbb{P}(\Delta(R_n)\leqslant c \log n).\]
Then the set $I$ of integers $n$ such that $p_n\geqslant 1-\eps$ has asymptotic lower density at least $1-\eps$.
\end{lemma}
\begin{proof}
As noted earlier, see~(\ref{eqn.slow}), by \cc{Lemma 28 of \cite{MSsizes}} there exists a constant $c_1=c_1(g,\eps)$ such that the set $I^* = I^*(g,\eps)$ of integers $n \geqslant 1$ for which
\begin{equation} \label{eqn.slowgrowth} \notag
\left|\cA_{n+1}^{g}\right| \leqslant c_1 \, (n+1)  \left|\cA_{n}^{g}\right|
\end{equation}
has lower density at least $1-\eps$. By Lemma \ref{lemma.pvlowerdensity} with $p=1-\eps/2$ there exists a constant $\alpha >0$ such that for $n \in I^*$ we have 
$\mathbb{P}(R_n \in \cL^\alpha_n) \geqslant 1-\eps/2 +o(1)$
(where $\cL^\alpha_n$ is the set of graphs $G$ on $[n]$ with at least $\alpha n$ leaves).
Let $b(n) = \lceil (8/\alpha) \log n \rceil$.
We have already seen that whp $R_n \in \cS$. 
Hence by Lemma~\ref{lem.Delta}, for $n\in I^*$
\[ \pr(\Delta(R_n) \geqslant b) \leq
\pr\left( (R_n \in \cL^\alpha  \cap \cS) \land (\Delta(R_n) \geqslant b)\right) + \pr(R_n \not\in \cL^\alpha  \cap \cS) \leqslant \eps/2 + o(1)\,. \]
Thus for all sufficiently large $n$ with $n \in I^*$ we have $\pr(\Delta(R_n)<b(n)) > 1-\eps$, which completes the proof.
\end{proof}

\subsection{Face sizes: proof of upper bound in Theorem~\ref{thm.facesize} (a)}
\label{subsec.facesizes}
In this section, we prove the upper bound in Theorem~\ref{thm.facesize} (a) on the maximum face size. We also state and prove a corresponding lower density version for the case when $g(n)$ is $O(n/\log n)$ in Lemma~\ref{lemma.facesizeLDupper}, which will be used to prove Theorem~\ref{thm.facesize} (b) in Section~\ref{degfacesize_LDversions}.

\begin{proof}[Proof of upper bound in Theorem~\ref{thm.facesize} (a)] 
Let $0<\alpha<\rho(\cP)$.
Let the constant $c$ satisfy 
$c > e/\alpha \, +1$, and let $\cB$ be the set of graphs $G \in \cA^g$ such that there is a relevant embedding of $G$ which has a face $F$ such that $|E(F)| \geqslant c \log v(G) +1$. Here $E(F)$ is the set of edges in the boundary of $F$, and the size of $F$ is at most $2|E(F)|$. It suffices to show that $|\cB_n|/|\cA_n^g| = o(1)$. Let $\cL_n^{\alpha}$ be the set of graphs $G$ on $[n]$ with at least $\alpha n$ leaves, and recall from Theorem~\ref{thm.leaves} (a) that $|\cL_n^{\alpha}\cap \cA_n^g|/|\cA^g_n| = 1+o(1)$. Let $\cB'_n = \cB_n \cap \cL_n^{\alpha}$. To show that  $|\cB_n|/|\cA^g_n| = o(1)$ it suffices to show that  $|\cB'_n|/|\cA^g_n| = o(1)$.

Let $n \in \N$ be large. Let $k= \log n$, and to avoid cluttering up notation assume that $k \in \N$. Let $G \in \cB'_n$. Fix a relevant embedding $\phi$ of $G$ in a
surface $S$ such that there is a face $F$ with $|E(F)| \geqslant c\, \log v(G) +1$.  List the (distinct) edges in $E(F)$ in (say) lexicographic order.
Pick an ordered list of $k+1$ distinct leaves $v_0, v_1,\ldots,v_k$ of $G$ (making at least $(\alpha n)_{(k+1)}= (1+o(1))\, (\alpha n)^{k+1}$ choices).  Let $G'$ be $G$ less the $k+1$ leaves $v_0, v_1,\ldots,v_k$.  We may simply delete these leaves from the embedding $\phi$ of $G$ to obtain an embedding $\phi'$ of $G'$ in the surface $S$ containing a face $F'$ with $|E(F')| \geqslant |E(F)|-(k+1) \geqslant (c-1) \log n$.

Pick a set of $k$ edges in $E(F')$, and list them in lexicographic order as $e_1,\ldots,e_k$.  We form the graph $G''$, and by amending $\phi'$ we form an embedding $\phi''$ of $G''$ in the surface $S$, as follows.  Insert $v_0$ in the open face $F'$; use $v_1,\ldots,v_k$ to subdivide the edges $e_1,\ldots,e_k$; and add the edges $v_0 v_i$ for $i=1,\ldots,k$ drawn in the open face $F'$ (so they are disjoint expect for meeting at $v_0$). Since  $\phi''$ is an embedding of $G''$ in the surface $S$ we have $G'' \in \cA^g_n$. The number of constructions of graphs $G''$ we have made is at least
\[(1+o(1))\, |\cB'_n|\, (\alpha n)^{k+1}\, \binom{(c-1) \log n}{k} \geqslant (1+o(1))\, |\cB'_n|\, (\alpha n)^{k+1}\, (c-1)^k\,.\]
If we guess the vertex $v_0$ in $G''$ then we can read off the (ordered) list $v_1,\ldots,v_k$; and thus $G''$ can be constructed at most $n^{k+2}$ times (after guessing $v_0$, we can just guess the original neighbours in $G$ of the leaves $v_0,\ldots,v_k$).  Hence
\[ |\cB'_n|\cdot (1+o(1))\,  (\alpha n)^{k+1}\, (c-1)^k \leqslant |\cA^g_n| \cdot n^{k+2}. \]
Thus
\[ \frac{|\cB'_n|}{|\cA^g_n|} \leqslant (1+o(1)) \,  (\alpha)^{-1}  \left( (c-1) \alpha  \right)^{-k}\, n = o(1) \]
since 
$(c-1)\alpha >e$ and so
$\left( (c-1) \alpha  \right)^{-k} \ll n^{-1}$.
Thus $\frac{|\cB'_n|}{|\cA^g_n|} =o(1)$, as required.
\end{proof}
We now give a corresponding lower density version for the upper bound on the maximum face size, for which our proof requires $g$ to be non-decreasing.
\begin{lemma}\label{lemma.facesizeLDupper}
Let $g(n) = O(n/\log n)$ be non-decreasing, and let $R_n \in_u \cA^g$. For any $\eps >0$ there is a constant $c>0$ such that the following holds. For $n\in \mathbb{N}$ let $p_n$ be the probability that the maximum face size in a relevant embedding of $R_n$ is at most $c\log n$. Then the set of integers $n$ such that $p_n\geqslant 1-\eps$ has lower density at least $1-\eps$.
\end{lemma}
\begin{proof}\m{do need non-decreasing} Let $0<\eps < 1$.
By Lemma~\ref{lemma.pvlowerdensity}, with $p=1-\tfrac12\eps$, there exists $\alpha > 0$ and $n_0 \in \mathbb{N}$ such that for all $n\geqslant n_0$ with $n\in I=I^*(g,\eps)$, the probability $\mathbb{P}(R_n \in \cL^{\alpha}_n) \geqslant 1-\tfrac12\eps$, where $\cL^{\alpha}_n$ is the set of graphs $G$ on $[n]$ with at least $\alpha n$ leaves. Let the constant $c_0$ satisfy $c_0 > e/\alpha\,+1$.  
Let $\cB$ be the set of graphs $G\in\cA^g$ such that there is a relevant embedding of $G$ which has a face $F$ such that $|E(F)|\geqslant c_0 \log v(G)+1$. Further, let $\cB_n' = \cB_n \cap \cL_n^{\alpha}$. Arguing as in the proof of the upper bound in Theorem~\ref{thm.facesize} (a), we may see that $\mathbb{P}(R_n\in \cB_n') = o(1)$. As a result, for all sufficiently large $n\in I$
\[ \mathbb{P}(R_n \in \cB_n) \leqslant \mathbb{P}(R_n \in \cB_n') + \mathbb{P}(R_n \not\in \cL_n^{\alpha}) \leqslant \tfrac12 \eps + \tfrac12 \eps =   \eps\,.\]
Finally, note that we may replace $c_0 \log v(G) +1$ by $c \log v(G)$ if $c >c_0$.
\end{proof}

\subsection{Maximum degree and face sizes: proofs of lower bounds in Theorems~\ref{thm.Delta2} (a) and~\ref{thm.facesize} (a)}
In this section, we prove the lower bounds on the maximum degree and maximum face size from Theorems~\ref{thm.Delta2} (a) and~\ref{thm.facesize} (a). To do so we first give a preliminary lemma concerning pendant appearances
of graphs with `few' automorphisms.  We say that a set $\cH$ of connected graphs \emph{has at most exponentially many automorphisms} if there is a constant $c$ such that $\aut(H) \leqslant c^{v(H)}$ for each graph $H \in \cH$.
\begin{lemma} \label{thm.fewautoms}
Let the genus function $g$ be $o(n/\log^3 n)$, and let $R_n \inu \cA^g$.
Let $\cH$ be a set of connected planar graphs which have at most exponentially many automorphisms.  Then for each $\eps>0$ there exists $\eta>0$ such that whp $R_n$ has at least $n^{1-\eps}$ pendant appearances of each graph $H \in \cH$ with at most $\eta \log n$ vertices.

In particular, there exists $\eta>0$ such that whp $R_n$ has a pendant appearance of the wheel with at least $\eta \log n$ vertices, and contains a path of length at least $\eta \log n$ in which each vertex has degree 2. 
\end{lemma}
\begin{proof} 
Let $\eps > 0$. Theorem 3.1 of~\cite{MaxDegree} states that for any set $\cH$ of connected planar graphs with at most exponentially many automorphisms, there exists $\eta>0$ such that whp $R_n \in_u \cP$ has at least $n^{1-\eps}$ pendant appearances of each graph $H\in \cH$ with at most $\eta \log n$ vertices. The same proof holds for $R_n \inu \cA^g$ 
as long as whp $R_n$ has a linear number of leaves, so in particular 
when $g(n)$ is $o(n/\log^3 n)$, by Theorem~\ref{thm.leaves} (a). Thus, Corollary 3.4 of~\cite{MaxDegree} also holds for all such classes, from which the lemma follows directly.
\end{proof}
The required lower bounds on the maximum degree and face size now follow in a straightforward manner.
\begin{proof}[Proof of Theorem~\ref{thm.Delta2} (a) and Theorem~\ref{thm.facesize} (a) lower bounds]
The lower bound in Theorem~\ref{thm.Delta2} (a) follows from Lemma~\ref{thm.fewautoms} since if a graph $G$ has as a subgraph a wheel with at least $\eta \log n$ vertices then of course $\Delta(G) \geqslant \eta \log n -1$.
Similarly, the lower bound in Theorem~\ref{thm.facesize} (a) follows from Lemma~\ref{thm.fewautoms}, since if $G$ contains a path of length at least $\eta \log n$ in which each vertex has degree 2 in $G$, then every embedding of $G$ in any surface contains a face with size at least $\eta \log n$ (and indeed with at least this number of distinct vertices in the boundary).
\end{proof}
We have now completed the proofs of Theorem~\ref{thm.Delta2} (a) and Theorem~\ref{thm.facesize} (a).

\subsection{Maximum degree and faces sizes: lower density versions and proofs of Theorems~\ref{thm.Delta2} (b) and~\ref{thm.facesize} (b)} \label{degfacesize_LDversions}
In this section, we prove Theorems~\ref{thm.Delta2}~(b) and~\ref{thm.facesize} (b). We do so by deducing these results from their corresponding lower density versions. The upper bounds in the lower density versions were already proven at the end of Sections~\ref{upperboundmaxdeg} and~\ref{subsec.facesizes}. The lower bounds in the lower density version are proven using the following lower density result corresponding to Lemma~\ref{thm.fewautoms}.
\begin{lemma}\label{lemma.pendcopiesH}
Let $g(n)$ be $O(n/\log n)$ and be non-decreasing, 
and let $R_n \inu \cA^g$.   Let $\cH$ be a set of connected planar graphs with at most exponentially many automorphisms. Let $\eps>0$ and $\delta > 0$. Then there exists $\eta > 0$ such that the set $I$ of integers $n$, such that the probability that $R_n$ has at least $n^{1-\delta}$ pendant appearances of each graph $H\in \cH$ with at most $\eta \log n$ vertices is at least $1-\eps$, has lower density at least $1-\eps$.
\end{lemma}
Part of the following proof is adapted from the proof of Theorem~3.1 in~\cite{MaxDegree}.
\begin{proof}\m{need non-decreasing}
We may assume that $0<\eps <1$.  Recall that $I^*(g,\eps)$ was defined at~(\ref{eqn.slow}), and $\cL^{\alpha}_n$ is the set of graphs $G$ on $[n]$ with at least $\alpha n$ leaves.
By Lemma~\ref{lemma.pvlowerdensity}, with $p=1-\tfrac12\eps$, there exists $\alpha > 0$ and $n_0 \in \mathbb{N}$ such that for all $n\geqslant n_0$ with $n\in I=I^*(g,\eps)$, the probability $\mathbb{P}(R_n \in \cL^{\alpha}_n) \geqslant 1-\tfrac12\eps$.
Let $c>0$ be such that each graph $H\in\cH$ with $h$ vertices has at most $c^h$ automorphisms. Further, let $\cH^j$ denote the set of graphs in $\cH$ with at most $j$ vertices.
Let $\delta > 0$. Let $\eta>0$ be sufficiently small and $n_1\geqslant n_0$ sufficiently large that, if we set $k=k(n)=\lfloor\eta\log n\rfloor$, then for all $n\geqslant n_1$ the following three conditions hold:
\begin{itemize}
\item $|\cH^k|\leqslant n^{\delta/5}$
\item $2\alpha^{-k} \leqslant n^{\delta/5}$
\item $2c^k\leqslant n^{\delta/5}$.
\end{itemize}
Using a double counting argument, we will show that for all sufficiently large $n\in I$ the probability that $R_n$ fails to have at least $n^{1-\delta}$ pendant appearances of each $H\in \cH^k$ is at most $1-\eps$.

Let $n_2>n_1$ be such that for each $n\geqslant n_2$ the following three conditions hold:
\begin{itemize}
    \item $1\leqslant k\leqslant \alpha n$
    \item $(\alpha n)_k\geqslant \tfrac12 (\alpha n)^k$
    \item $(n^{1-\delta}+k)\,c^k\leqslant n^{1-4\delta/5}$.
\end{itemize}
Let $n\geqslant n_2$ and consider any graph $H\in \cH^k$ with $2 \leqslant h\leqslant k$ vertices. Fix a labelled copy of $H$ on vertices $1, \dots, h$ rooted 
at vertex 1.  Following~\cite{MaxDegree} we show the lower bound for rooted pendant appearances, without further comment. Let
\[\cB_n^H = \{G\in\cA^g_n \mid \ell(G)\geqslant \alpha n \mbox{ and } G \mbox{ has at most } n^{1-\delta} \mbox{ pendant appearances of } H\}.\]
We now show that from each graph $G\in\cB^H_n$ we can construct many graphs in $\cA^g_n$ with little double counting, as follows.

Let $G\in\cB^H_n$ and choose an arbitrary ordered list of $h$ pendant vertices $u_1, \dots, u_h$. Delete the edge incident with $u_j$ for each $j=2,\dots,h$ and form a pendant appearance of $H$ rooted at vertex $u_1$ with vertex $u_j$ corresponding to vertex $j$ of $H$ for $j=1,\dots,h$.  From $G$ we make at least $(\alpha n)_h \geqslant \tfrac12 (\alpha n)^h$ constructions, and each graph we construct is in $\cA^g_n$ since $H$ is planar.

How often is each graph constructed? To get back to $G$, we guess the constructed pendant appearance of $H$ -- there are at most $n^{1-\delta}+h$ of them.
\m{should amplify but ..}
Now we know the vertex $u_1$ and the set $\{u_2, \dots, u_h\}$. Since there are at most $c^h$ automorphisms of $H$, there are at most $c^h$ possible choices on the order of $u_2, \dots, u_h$. Finally, guess the vertices that $u_2, \dots, u_h$ were originally attached to -- there are at most $n^{h-1}$ choices. Thus, each graph is constructed at most $(n^{1-\delta}+h)\cdot c^h \cdot n^{h-1} \leqslant n^{h-4\delta/5}$ times.
Therefore
\[ |\cB^H_n| \cdot \tfrac12 (\alpha n)^h \leqslant |\cA^g_n| \cdot n^{h-4\delta/5}.\]
Thus
\[\mathbb{P}(R_n\in \cB_n^H) \leqslant 2n^{h-4\delta/5}(\alpha n)^{-h} \leqslant n^{-3\delta/5}\]
and so
\[\mathbb{P}(R_n\in \bigcup_{H\in\cH^k}\cB_n^H) \leqslant \sum_{H\in\cH^k} \mathbb{P}(R_n\in \cB_n^H) \leqslant n^{-2\delta/5}.\]
Hence
\begin{eqnarray*}
&& \pr(R_n \mbox{ has } < n^{1-\delta} \mbox{ pendant appearances of } H \mbox{ for some } H \in \cH^k)\\
& \leqslant &
\pr(R_n \in \bigcup_{H\in\cH^k}\cB_n^H) + \mathbb{P}(R_n \not\in \cL^\alpha_n)\\
&\leqslant & 
n^{-2\delta/5} + \tfrac12 \eps \;\;
\leqslant \;\; \eps
\end{eqnarray*}
for $n$ sufficiently large. So, for all sufficiently large $n\in I$, 
with probability at least $1-\eps$ the random graph $R_n$ has at least $n^{1-\delta}$ pendant appearances of each $H\in\cH^k$.
\end{proof}
From Lemma~\ref{lemma.pendcopiesH} we can directly deduce that there must be pendant appearances of the wheel with at least $\eta \log n$ vertices and a path of length at least $\eta \log n$ with sufficiently high probability as long as $n\in I$.
\begin{lemma}\label{lemma.degreefacelowerdensity}\m{need non-decreasing}
Let $g(n)=O(n/\log n)$ be non-decreasing and let $R_n \inu \cA^g$.  Let $\eps>0$.
Then there exists $\eta>0$ such that the following holds.
Let $p_n(\mbox{wheel})$ be the probability that $R_n$ has a pendant appearance of the wheel with at least $\eta \log n$ vertices, and let $p_n(\mbox{path})$ be the probability that $R_n$ contains a path of length at least $\eta \log n$ in which each vertex has degree 2. Then the set $I$ of integers $n$ such that $p_n(wheel) \geqslant 1-\eps$ and $p_n(path) \geqslant 1-\eps$ has lower density at least $1-\eps$.
\end{lemma}

Using this lemma, we can give a lower density version for the lower and upper bounds on the maximum degree.
\begin{lemma}\label{lemma.maxdeglowerdensity}
Let $g(n)=O(n/\log n)$ be non-decreasing and let $R_n\in_u\cA^g$.
Then for any $\eps>0$ there are constants $0< c_1 < c_2$ such that the following holds. For $n \in \N$ let  \[p_n = \pr \left(c_1 \log n \leqslant \Delta(R_n) \leqslant c_2 \log n \right) \,. \] Then the set $I$ of integers $n$ such that $p_n \geqslant 1- \eps$ has lower density at least $1-\eps$.
\end{lemma}
\begin{proof}\m{need non-decreasing}
We prove this lemma in two parts. 
Consider first the lower bound. By Lemma~\ref{lemma.degreefacelowerdensity} for the wheel, there exists $c_1>0$ such that, if $I_1$ is the set of $n \in \N$ with $\pr(\Delta(R_n) \geqslant c_1 \log n) \geqslant 1-\tfrac12 \eps$, then $I_1$ has lower density at least $1-\tfrac12 \eps$.
For the upper bound, by Lemma~\ref{lemma.maxdegLDupperbound}, there exists $c_2>0$ such that, if $I_2$ is the set of $n \in \N$ with $\pr(\Delta(R_n) \leqslant c_2 \log n) \geqslant 1 - \tfrac12 \eps$, then $I_2$ has lower density at least $1- \tfrac12 \eps$.
But now the set $I = I_1 \cap I_2$ of integers has lower density at least $1-\eps$, and $\pr \left(c_1 \log n \leqslant \Delta(R_n) \leqslant c_2 \log n \right) \geqslant 1-\eps$ for each $n \in I$.
\end{proof}
\begin{proof}[Proof of Theorem~\ref{thm.Delta2} (b)]\m{need non-decreasing}
This follows directly from the lower density version Lemma~\ref{lemma.maxdeglowerdensity}.
\end{proof}

Similarly, we give a lower density version for the lower and upper bounds on the maximum face size.
\begin{lemma} \label{facesize_lowerdensity}
Let $g(n) = O(n/\log n)$ be non-decreasing and let $R_n\in_u \cA^g$. Then for any $\eps>0$ there are constants $0<c_1<c_2$ such that the following holds. For $n\in \mathbb{N}$ let $p_n$ be the probability that the maximum face size in a relevant embedding is at least $c_1\log n$ and at most $c_2 \log n$.
Then the set $I$ of integers $n$ such that $p_n \geqslant 1-\eps$ has lower density at least $1-\eps$.
\end{lemma}
\begin{proof}\m{need non-decreasing}
As for Lemma~\ref{lemma.maxdeglowerdensity}, we prove this lemma in two parts, starting with the lower bound. As we noted before, 
if $G$ has a path of at least $k$ degree-two vertices then every embedding of $G$ has a face of size at least $k$.
Thus by Lemma~\ref{lemma.degreefacelowerdensity} (for the path of degree two vertices) there exists $c_1>0$ such that, if $F_n$ is the event that every embedding of $R_n$ has a face of size at least $c_1 \log n$, and $I_1$ is the set of $n \in \N$ with $\pr(F_n) \geqslant 1-\tfrac12 \eps$, then $I_1$ has lower density at least $1-\tfrac12 \eps$.
For the upper bound, by
Lemma~\ref{lemma.facesizeLDupper}, there exists $c_2>0$ such that, if $F'_n$ is the event that in every relevant embedding of $R_n$ every face has length at most $c_2 \log n$, and $I_2$ is the set of $n \in \N$ with $\pr(F'_n) \geqslant 1-\tfrac12 \eps$, then $I_2$ has lower density at least $1-\tfrac12 \eps$.
But now the set $I = I_1 \cap I_2$ of integers has lower density at least $1-\eps$, and $\pr (F_n \cap F'_n) \geqslant 1-\eps$ for each $n \in I$.
\end{proof}
\begin{proof}[Proof of Theorem~\ref{thm.facesize} (b)]\m{need non-decreasing}
This follows directly from the lower density version Lemma~\ref{facesize_lowerdensity}.
\end{proof}

We have seen that when $g(n) = o(n/\log^3n)$,
for $R_n \inu \cA^g$ the maximum degree and the maximum face size of a relevant embedding are both $\Theta(\log n)$ whp. Further, when $g(n) = o(n/\log n)$, for both the maximum degree and the maximum face size,
the limsup of the probability that it is $\Theta(\log n)$ is still near 1.
By Theorem~\ref{thm.edges}, assuming $g(n)=o(n^2)$, whp $R_n$ has average degree at least $1+ g/n$; and thus if $g(n) \gg n \log n$ we have $\Delta(R_n) \gg \log n$ whp.
We conjecture that, for $g(n)=O(n^2)$, the maximum degree is 
$\Theta(\log n + g/n)$ whp.
\m{what about max face size for large $g$?}

\section{Planarity of the fragment of $R_n\inu\cA^g$, proof of Theorem~\ref{thm.Fragplanar}}
\label{sec.frag}
In this section we prove Theorem~\ref{thm.Fragplanar} on the planarity of the fragment (the subgraph induced on the vertices not in the largest component) of $R_n$.
We use three preliminary lemmas.  For any connected graph $K$, let $\comp(K)$ be the class of graphs with a component isomorphic to $K$.  We first consider the orientable case.

\begin{lemma} \label{lem.ori}
Let $g=g(n) \geqslant 0$ satisfy $g \ll n^2$, and let $K$ be a (fixed) $k$-vertex nonplanar connected graph. Then 
 \[ \max_{0 \leqslant h \leqslant g(n)} \{ | \comp(K) \cap \cO\cE^h_n| / |\cO\cE^h_n| \} \rightarrow 0 \;\;\; \mbox{ as } \; n \to \infty\,. \]
\end{lemma}
\begin{proof}
Suppose that $K$ has $k$ vertices.  For each $h \geqslant 2$
\[ | \comp(K) \cap \cO\cE^h_n | \leqslant (n)_k \, |\cO\cE_{n-k}^{h-2}| .\]
By equations~(\ref{eqn.lbasymp}) and~(\ref{eqn.growth}), there is an $n_0$ such that, for each $n \geqslant n_0$ and $2 \leqslant h \leqslant g(n)$,
\[ |\cO\cE_n^{h}| \geqslant \frac{n^2}{7(n\!+\!h)}\, |\cO\cE_n^{h-2}| \geqslant \frac{n^2}{7(n\!+\!g)}\, 2^k(n\!-\!1)_{k} \, |\cO\cE_{n-k}^{h-2}| \geqslant \frac{n^2}{7(n\!+\!g)}\, (n)_{k} \, |\cO\cE_{n-k}^{h-2}| . \]
Hence, for each $n \geqslant n_0$ and $2 \leqslant h \leqslant g(n)$, \[ |\comp(K) \cap \cO\cE^h_n | / |\cO\cE_{n}^{h}| \leqslant (n)_k \, |\cO\cE_{n-k}^{h-2}| / |\cO\cE_{n}^{h}| \leqslant \frac{7(n+g)}{n^2}\,  = o(1)\,, \]
which completes the proof.
\end{proof}

For the non-orientable case, we first need to consider pendant appearances of a nonplanar connected graph $K$. By the additivity of Euler genus on blocks, for each $h \geqslant 0$, each graph $G \in \cN\cE_n^h$ has at most $h$ disjoint pendant appearances of $K$. We need a better bound for random graphs in $\cN\cE_n^h$.

\begin{lemma} \label{lem.pendH}
For each $n \geqslant 1$ and $h \geqslant 0$ let $R_n^h \inu \cN\cE_n^h$. Let $g=g(n)$ satisfy $0 \leqslant g(n) \ll n^2$. Let $t=t(n)$ satisfy $t^2 \gg n+g$ and $t \ll n$.  Let $K$ be a $k$-vertex nonplanar connected graph.  Then 
\begin{equation} \label{eqn.pend} \nonumber
  \max_{0 \leqslant h \leqslant g(n)} \pr(\pend(R_n^h,K) \geqslant t)
  \rightarrow 0 \;\;\; \mbox{ as } n \to \infty.
\end{equation}
\end{lemma}
\begin{proof}
Let $u = u(t)$ be either $\lceil t/k \rceil$ or $\lceil t/k \rceil -1$, whichever is even.  If $\pend(G,K) \geqslant t$ then $G$ has at least $u$ (vertex) disjoint pendant appearances of $K$. Let $n_0 \geqslant 2$ be sufficiently large that $u < n/k$ for all $n \geqslant n_0$.
For each $n \geqslant n_0$ and $t \leqslant h \leqslant g(n)$
\begin{eqnarray*}
&&|\{ G \in \cN\cE_n^{h}: \pend(G,K) \geqslant t \}|\\
& \leqslant &
|\{ G \in \cN\cE_n^{h}: G \mbox{ has} \geqslant u \mbox{ disjoint pendant appearances of } K) \}|\\
& \leqslant &
\frac{(n)_{u k}}{u!\, (\aut\, K)^u} (k n)^u \, |\cE_{n-u k}^{h-u}| \;\; \leqslant \;\; \frac{(n)_{u k} (k n)^u}{u!}  \, |\cE_{n-u k}^{h-u}|\,,
\end{eqnarray*}
since $\aut\, K \geqslant 1$.
Also, by equation~(\ref{eqn.lbasymp}) there is an $n_1 \geqslant n_0$ such that for all $n \geqslant n_1$ and $2 \leqslant h \leqslant g(n)$ we have
\[|\cN\cE_{n}^{h}| \geqslant \frac{n^2}{7(n+h)} \, |\cE_{n}^{h-2}|\, . \]
Using also equation~(\ref{eqn.growth}), for all $n \geqslant n_1$ and $t \leqslant h \leqslant g(n)$ we have
\[ |\cN\cE_n^{h}| \geqslant \left(\frac{n^2}{7(n+h)} \right)^{u/2} |\cE_n^{h-u}|\geqslant \left(\frac{n^2}{7(n+g)} \right)^{u/2} \, 2^{uk}(n-1)_{uk} \,  |\cE_{n-uk}^{h-u}|
\geqslant \left(\frac{n^2}{7(n+g)} \right)^{u/2} \, (n)_{uk} \,  |\cE_{n-uk}^{h-u}|. \]
Hence, for all $n \geqslant n_1$ and $t \leqslant h \leqslant g(n)$ (noting that $u! \geqslant (u/e)^u$)
\begin{eqnarray*}
\pr(\pend(R_n^h,K) \geqslant t)
&=&
|\{ G \in \cN\cE_n^{h}: \pend(G,K) \geqslant t \}|\, / \, |\cN\cE_n^{h}| \\
& \leqslant &
\frac{(n)_{u k} (k n)^u }{u!} 
\left(\frac{7(n+g)}{n^2} \right)^{u/2} \frac1{ (n)_{uk}} \\
&=&
\frac{k^u \, (7(n+g))^{u/2}}{u!}\\
& \leqslant &
\left( \frac{\sqrt{n+g}}{u} \;
\sqrt{7}\, k e \right)^{u} \;\; \leqslant \;\; 2^{-t}
\end{eqnarray*}
for $n$ sufficiently large, since $u \gg \sqrt{n+g}$.
Hence $\pr(\pend(R_n^h,K) \geqslant t) \leqslant 2^{-t}$ for $n$ sufficiently large, which completes the proof.
\end{proof}
We now give a lemma for the nonorientable case corresponding to Lemma~\ref{lem.ori} (which handles the orientable case).

\begin{lemma} \label{lem.nonori}
Let $g=g(n) \geqslant 0$ satisfy $g(n) \ll n^2$, and let $K$ be a $k$-vertex nonplanar connected graph. Then
\begin{equation} \label{eqn.Knonori} \nonumber
 \max_{0 \leqslant h \leqslant g(n)} \{ | \comp(K) \cap \cN\cE^h_n| / |\cN\cE^h_n| \} \rightarrow 0 \;\;\; \mbox{ as } n \to \infty\,.
\end{equation}
\end{lemma}
\begin{proof}
As in Lemma~\ref{lem.pendH}, let $t=t(n)$ 
satisfy $t^2 \gg n+g(n)$ and $t \ll n$.  Let $\cD_n^h$ be the set of graphs $G \in \cN\cE_n^{h}$ with $\pend(G,K) \leqslant t$. By Lemma~\ref{lem.pendH}, if we let
\begin{equation} \label{eqn.cD}
 \eps(n) = 1-\min_{0 \le h \le g(n)} |\cD_n^h| / |\cN\cE_n^h| 
\end{equation}
then $\eps(n)=o(1)$.
From each graph $G \in \comp(K) \cap \cD_n^h$ we may construct at least $k(n-k)$ graphs $G' \in \cN\cE_n^h$ by adding an edge between a copy of $K$ and a vertex in the rest of the graph.
For each graph $G'$ constructed we have $\pend(G',K) \leqslant t+k$; and so $G'$ can be constructed at most $t+k$ times. Hence
\[ |\comp(K) \cap \cD_n^h| \cdot k(n-k) \leqslant |\cN\cE_n^h| \cdot (t+k) .\]
Thus, recalling the definition~(\ref{eqn.cD}),
\[ \frac{|\comp(K) \cap \cN\cE_n^h|}{|\cN\cE_n^h|} \leqslant \frac{t+k}{k(n-k)} + \eps(n) = o(1) \, ;\]
and the lemma follows.
\end{proof}

We may now use Lemmas~\ref{lem.ori} and~\ref{lem.nonori} to complete the proof of Theorem~\ref{thm.Fragplanar}.

\begin{proof}[Proof of Theorem~\ref{thm.Fragplanar}]
Let $g(n) = n \log n$. (Any $g$ such that $n \ll g(n) \ll n^2$ would do.)  By part (c) of Theorem~\ref{thm.conn}, since $g \gg n$,
\begin{equation} \label{eqn.disc}
 \max_{S:\, eg(S)\ge g(n)} \pr( R^S_n \mbox{ is disconnected}) \; = \; o(1).
\end{equation}
Note that Theorem~\ref{thm.conn} has not been proven yet, but the proof of part (c) of this result does not use Theorem~\ref{thm.Fragplanar}. Now fix $\eps_0>0$.  By~(\ref{eqn.disc}) it suffices to show that, for $n$ sufficiently large,
\begin{equation} \label{eqn.egsmall}
 \max_{S :\, eg(S) \le g(n)} \pr(\Frag(R^S_n) \mbox{ is nonplanar}) \; \le \; \eps_0 \, .
\end{equation}
Fix $\ell = \lceil 4/\eps_0 \rceil$. For each $n \geqslant 1$ and each surface $S$, since $\E[\frag(R_n^S)] < 2$ (see Lemma \ref{lem.frag}), 
we have $\pr(\frag(R_n^S) \geqslant \ell) \leqslant \eps_0/2$. 
Let $H_1,H_2,\ldots,H_j$ list all the unlabelled nonplanar connected graphs on at most $\ell -1$ vertices.
By Lemmas~\ref{lem.ori} and~\ref{lem.nonori}, for each $i=1, \ldots,j$ there exists $\eps_i(n)=o(1)$ such that for each $n \geqslant 1$
\[ \max_{S:\, eg(S) \leqslant g(n)} \pr(R_n^S \mbox{ has a component isomorphic to } H_i) \leqslant \eps_i(n). \]
Thus, for each $n \geqslant 1$ 
\begin{eqnarray*}
&& \max_{S:\, eg(S) \leqslant g(n)} \pr(\Frag(R_n^S)\mbox{ is nonplanar}) \\
&\leqslant &
  \max_{S:\, eg(S) \leqslant g(n)} \left( \pr(\frag(R_n^S) \geqslant \ell) + \sum_{i=1}^j \pr(R_n^S \mbox{ has a component isomorphic to } H_i) \right)\\ 
& \leqslant &
  \eps_0/2 + \sum_{i=1}^j \eps_i(n).
\end{eqnarray*}
Thus~(\ref{eqn.egsmall}) holds for $n$ sufficiently large, and this completes the proof.
\end{proof}

\section{Probability that $R_n\inu\cA^g$ is connected, proof of Theorem~\ref{thm.conn}}
\label{sec.conn}
In this section we prove Theorem~\ref{thm.conn} on the probability that $R_n \inu \cA^g$ is connected.
After proving part~(a), we justify the statement immediately after Theorem~\ref{thm.conn} (using Theorem~\ref{thm.Fragplanar}), then prove parts (b) and (c). Note that in proving Theorem~\ref{thm.conn} we do not use results proved earlier in this paper, apart from some preliminary results from Section~\ref{sec.back}.

%

\begin{proof}[Proof of Theorem~\ref{thm.conn} (a)]
Assume that $g(n) \ll n/\log^3 n$.
Recall that $\cC$ is the set of connected planar graphs, with exponential generating function $C(x)$.  Let $H \in \cC$, let $h=v(H)$, and let $\alpha_H= h \rho^h/\aut H$; and let $0 <\eps<1$. 
By equation~(\ref{eqn.pend-g}) we have $\pend(R_n,H) \geqslant (1-\eps) \alpha_H n$ wvhp.  It now follows that the conditions of Lemma 6.1 (a) in~\cite{PendAppComp} hold (with $\cA$ in the lemma as $\cA^g$ and $\cC$ in the lemma as here), and so
\[ \limsup_{n \to \infty}\, \pr(\Frag(R_n,\cP) = \emptyset) \leqslant e^{-C(\rho(\cP))} = p^*\,.\]
The required result now follows (without using Theorem~\ref{thm.Fragplanar}).
\end{proof}

Now let us justify the statement immediately after Theorem~\ref{thm.conn}.
Given a class $\cB$ of graphs, let $\con(\cB)$ be the set of connected graphs in $\cB$.
For every surface $S$, both $|\cE^S_n|/ n |\cE^S_{n-1}|$ and $|\con(\cE^S_n)|/ n |\con(\cE^S_{n-1})|$ tend to $\gamma_\cP$ as $n \to \infty$.
Conjecture 14 of \cite{MSsizes} would imply that, for each $\eps >0$ there is an $n_0$ such that for each $n\geqslant n_0$ and each surface~$S$, we have 
$\,\left|\con(\cE_{n}^{S})\right| / \left|\con(\cE_{n-1}^{S})\right| \geqslant (1-\eps) \, |\con(\cP_{n})|/|\con(\cP_{n-1})|\,$ (and similarly in the not-necessarily-connected case).  Let $\cA^g$ be either $\cO\cE^g$ or $\cN\cE^g$. (The case when $\cA^g$ is $\cE^g$ will follow easily from these cases.) 
Then, for each $\eps >0$ there would be an $n_0$ such that for each $n \geqslant n_0$
\begin{equation}
\label{eqn.conjforus}
\frac{|\con(\cA^g_n)|}{|\con(\cA^{g(n)}_{n-1})|} \geqslant  (1-\eps) \,\frac{|\con(\cP_n)|}{|\con(\cP_{n-1})|} .
\end{equation}
Let us assume temporarily that~(\ref{eqn.conjforus}) holds.

Let $\eta>0$. Let $k=\lfloor 2/\eta \rfloor$. By Lemma~\ref{lem.frag} (b) we have
$\pr(\frag(R_n) >k) < \eta$. 
For each planar graph $H$ with $1 \leqslant v(H) =h \leqslant k$, and each $n >2k$
\[ \pr(\Frag(R_n) =H) = \frac{(n)_h}{\aut\, H} \frac{|\con(\cA^{g(n)}_{n-h})|}{|\cA^{g}_{n}|}.\]
Let $R^{\cP}_n \inu \cP$.
Then as above
\[ \pr(\Frag(R^{\cP}_n) =H) = \frac{(n)_h}{\aut\,H} \frac{|\con(\cP_{n-h})|}{|\cP_{n}|}.\]
Hence
\begin{eqnarray*}
\pr(\Frag(R_n) =H)
&=&
\pr(\Frag(R^{\cP}_n) =H)\, \frac{|\con(\cA^{g(n)}_{n-h})|}{|\con(\cA^{g}_{n})|} \frac{|\con(\cA^{g}_{n})|}{|\cA^{g}_{n}|}\, \frac{|\con(\cP_{n})|}{|\con(\cP_{n-h})|} \frac{|\cP_{n}|}{|\con(\cP_{n})|}\\
&=&
\pr(\Frag(R^{\cP}_n) =H)\, \frac{|\con(\cA^{g(n)}_{n-h})|}{|\con(\cA^{g}_{n})|}\, \frac{|\con(\cP_{n})|}{|\con(\cP_{n-h})|}\, \frac{\pr(R_n \mbox{ is connected})}{\pr(R^{\cP}_n \mbox{ is connected})}\,.
\end{eqnarray*}
We may assume that the last factor is at most 1.  Thus, by the assumed equation~(\ref{eqn.conjforus}) with $\eps>0$ sufficiently small that $(1-\eps)^{-k} \leqslant 1+\eta$, for all $n \geqslant n_0 +k$
\begin{eqnarray*}
\pr(\Frag(R_n) =H) & \leqslant & \pr(\Frag(R^{\cP}_n) =H)\, \frac{|\con(\cA^{g(n)}_{n-h})|}{|\con(\cA^{g}_{n})|}\, \frac{|\con(\cP_{n})|}{|\con(\cP_{n-h})|}\\
& \leqslant &
\pr(\Frag(R^{\cP}_n) =H)\, (1-\eps)^{-k} \; \leqslant \; (1 +\eta)\, \pr(\Frag(R^{\cP}_n) =H).
\end{eqnarray*}
Hence
\[ \pr((1 \leqslant \frag(R_n) \leqslant k) \land (\Frag(R_n) \mbox{ is planar})) \leqslant (1+\eta) \,
\pr(1 \leqslant \frag(R^{\cP}_n) \leqslant k).
\]
But by Theorem~\ref{thm.Fragplanar}, for $n$ sufficiently large we have $\pr(\Frag(R_n) \mbox{ is nonplanar}) \leqslant \eta$, and then
\begin{eqnarray*}
&& \pr(R_n \mbox{ not connected})\\ 
& \leqslant &
\pr((1 \leqslant \frag(R_n) \leqslant k) \land (\Frag(R_n) \mbox{ is planar})) +
\pr(\frag(R_n) >k) + \pr(\Frag(R_n) \mbox{ is nonplanar})\\
& \leqslant &
(1+\eta) \,\pr(R^{\cP}_n \mbox{ not connected}) + 2 \eta.
\end{eqnarray*}
Hence
\[ \liminf_{n \to \infty}\, \pr(R_n \mbox{ is connected}) \geqslant p^*, \]
as required (assuming that equation~(\ref{eqn.conjforus}) holds).
\medskip

We now state and prove the `lower density' version of Theorem~\ref{thm.conn} (b), from which we can directly deduce Theorem~\ref{thm.conn} (b).
\begin{lemma}\label{lemma.connectedLD}
Let $g(n)=O(n/\log n)$ be non-decreasing and let $R_n \in_u \cA^g$.
For every $\eps>0$ there is an $\eta>0$ such that the following holds. The set I of integers $n \geqslant 3$ such that $\mathbb{P}(R_n \mbox{ is connected}) \leqslant 1-\eta$ has lower density at least $1-\eps$.
\end{lemma}
\begin{proof}\m{need non-decreasing}
By Lemma~\ref{lem.frag} (a), $\pr(R_n \mbox{ is connected}) \geqslant \frac1{e}$.
By Lemma~\ref{lemma.pvlowerdensity}, 
there exist $\alpha>0$ and a set $I \subseteq \mathbb{N}$ of integers $n \geqslant 3$ of lower density at least $1-\eps$ such that $\pr(\ell(R_n) \geqslant \alpha n) \geqslant \frac43 - \frac1{e}$ for all $n \in I$. 
Thus 
\[ \pr\big((\ell(R_n) \geqslant \alpha n) \land (R_n \mbox{ is connected})\big)  \geqslant \tfrac13 \;\; \mbox{ for } n \in I .\]
Let $n \in I$. From each connected graph $G \in \cA^g_n$ with $\ell(G) \geqslant \alpha n$, by deleting the edge incident to a leaf, we may construct at least $\alpha n$ disconnected graphs $G' \in \cA^g_n$, and each graph $G'$ is constructed at most $n$ times (guess the neighbour in $G$ of the unique isolated vertex in $G'$). Hence
\[ |\{G \in \cA^g_n : (\ell(G) \geqslant \alpha n) \land (G \mbox{ is connected})\}| \cdot \alpha n \leqslant |\{G \in \cA^g_n : G \mbox{ disconnected}\}| \cdot n\,. \]
Thus for each $n \in I$
\[\pr(R_n \mbox{ is disconnected})
\geqslant \frac{|\{G \in \cA^g_n :
(\ell(G)\geqslant \alpha n) \land (G \mbox{ is connected})  \}| \cdot \alpha}{|\cA^g_n|} \geqslant \tfrac13 \alpha\,, \]
so we may take $\eta= \frac13 \alpha$ to complete the proof.
\end{proof}
\begin{proof}[Proof of Theorem~\ref{thm.conn} (b)]\m{need non-decreasing}
Follows directly from Lemma~\ref{lemma.connectedLD}.
\end{proof}

\medskip

Recall that for a graph $G$ we denote by $\Frag(G)$ the fragment of $G$ (the unlabelled subgraph of $G$ induced on the vertices not in the largest component), and by $\frag(G)$ the number of vertices in $\Frag(G)$. To prove part (c) of Theorem~\ref{thm.conn}, we need two preliminary lemmas. 
\begin{lemma} \label{lem.baddprob}
For every $\eps>0$ there is a $c=c(\eps)$ such that the following holds.  Let $\cD$ be a bridge-addable class of graphs which is closed under deleting all the edges of a component.
Then for each $n$ such that $\cD_n \neq \emptyset$, the random graph $R_n \in_u \cD$ satisfies
\begin{equation} \label{eqn.connvfrag1}
\pr(R_n \mbox{ is connected}) \geqslant 1-\eps - c \, \pr(\frag(R_n)=1).
\end{equation}
\end{lemma}
Observe that for example we could take the class $\cD$ in the lemma as 
the class of perfect planar graphs. This lemma appears also in~\cite{Colinnewbook}.
\begin{proof}
For each $j \in \N$ let $S_j$ denote the unlabelled $j$-vertex graph with no edges (S is for Stable).
Let $j \geqslant 2$ and let $n \geqslant j+2$. From each graph $G \in \cD_n$ with $\Frag(G) = S_j$ we may
construct $j(n-j)$ graphs $G' \in \cD_n$ with $\Frag(G') = S_{j-1}$ by picking an isolated vertex and joining it to the giant component; and each graph $G'$ is constructed at most $n-j$ times (since $G'$ has at most $n-j$ leaves).  Thus
\[ \pr(\Frag(R_n) = S_j)\, j(n-j) \leqslant \pr(\Frag(R_n) = S_{j-1}) \, (n-j);\]
so
\[ \pr(\Frag(R_n) = S_j) \leqslant (1/j)\, \pr(\Frag(R_n) = S_{j-1});\]
and thus
\[ \pr(\Frag(R_n) = S_j) \leqslant (1/j!) \, \pr(\frag(R_n)=1). \]
Let $\cB_j$ be the set of $j$-vertex (unlabelled) graphs
which can appear as $\Frag(G)$ for some graph $G \in \cD$.
For each graph $H \in \cB_j$, since there are at most $j!$ graphs on $[j]$ isomorphic to $H$,
\[ \pr(\Frag(R_n) = H) \leqslant j! \; \pr(\Frag(R_n) = S_j) \leqslant \pr(\frag(R_n)=1). \]
Also of course $|\cB_j| \leqslant 2^{\binom{j}{2}}$.  Thus for each $k \geqslant 1$
\begin{eqnarray*} 
  \pr(1 \leqslant \frag(R_n)\leqslant k) 
& = &
  \sum_{j=1}^k \sum_{H \in \cB_j} \pr(\Frag(R_n) = H) \\
& \leqslant & 
\sum_{j=1}^k |\cB_j| \cdot \pr(\frag(R_n)=1)\\
& = &
 c_k \: \pr(\frag(R_n)=1)
\end{eqnarray*}
where $c_k =\sum_{j=1}^k |\cB_j| \;\;(\leqslant \sum_{j=1}^{k} 2^{\binom{j}{2}})$.
Now let $0<\eps \leqslant 1$, and let $k=\lfloor 2/\eps \rfloor$. Then $\pr(\frag(R_n) >k)<\eps$, since $\E[\frag(R_n)] < 2$ by Lemma \ref{lem.frag} (b).
Thus by the above with $c=c_k$
\[ \pr(R_n \mbox{ is connected}) = 1-\pr(\frag(R_n) >k) - \pr(1 \leqslant \frag(R_n)\leqslant k) \geqslant 1-\eps - c \, \pr(\frag(R_n)=1), \] 
as required.  This completes the proof of Lemma~\ref{lem.baddprob}.
\end{proof}
%

Given a genus function $g$, the \emph{fixed surface growth ratio} $\fsgr(\cA^g,n)$ is 
defined by
\[ \fsgr(\cA^g, n) = \frac{|\cA_n^{g}|}{ n\, |\cA_{n-1}^{g(n)}|}\, . \]
Let us first note a lower bound on this quantity which is useful when $g$ is large (compare with~(\ref{eqn.growth})).
From each graph $G \in \cA^{g(n)}_{n-1}$, by adding vertex $n$ and making it isolated or adjacent to a single vertex or adjacent to both ends of an edge of $G$ or subdividing an edge of $G$, we can construct  $1+(n-1)+e(G)+e(G)$ graphs $G'$ in $\cA^g_n$, and all these graphs are distinct.  Thus, letting 
$\hat{g}(n)=g(n+1)$ for $n \in \N$ and $S_{n} \inu \cA^{\hat{g}}$,
\[ |\cA^g_n| \geqslant \sum_{G \in \cA^{g(n)}_{n-1}} (n+2e(G))
= |\cA^{g(n)}_{n-1}|\, (n+2\,\E[e(S_{n-1})]) \,,\]
and so
\begin{equation} \label{eqn.fsgr1}
\fsgr(\cA^g, n) \geqslant 1+ (2/n)\, \E[e(S_{n-1})].
 \end{equation}
(These inequalities together with Theorem~\ref{thm.edges} on numbers of edges, yield an improvement on the bound in inequality~(\ref{eqn.growth}).
If $\eps>0$ then for all sufficiently large $n$ and all $0 \leq h \leq (\tfrac12 -\eps) n^2$ we have 
 \[ |\cA^h_{n+1}| / |\cA^h_n|  \geq  (1- \eps) (3n + 2h)\,. \;\; \big) \]
\begin{lemma} \label{lem.conn-and-fsgr}
Let $g$ be a genus function, and let $R_n \in_u \cA^g$.
Then $R_n$ is connected whp  if and only if  $\, \fsgr(\cA^g, n) \to \infty$  as $n \to \infty$.
\end{lemma}
\begin{proof}
We claim that
\begin{equation} \label{claim.conn}
\frac1{e} \cdot \frac1{\fsgr(\cA^g,n)} \leqslant \pr(\frag(R_n)=1) \leqslant  \frac1{\fsgr(\cA^g,n)}.
\end{equation} 
Let us establish this claim. First note that
\[ \pr(\frag(R_n)=1) = 
\frac{\left|\{ G \in \cA_n^g : \frag(G)=1 \}\right|}{ |\cA_n^g|}
= \frac{ n \, \left| \{G \in \cA_{n-1}^{g(n)}: G \mbox{ connected}\}\right|}{ |\cA_n^g|}.\]
Thus
\[ \pr(\frag(R_n)=1)  \leqslant \frac{n \, | \{G \in \cA_{n-1}^{g(n)} \}|}{|\cA_n^g|}
= \frac1{\fsgr(\cA^g,n)}. \]
Also, by Lemma~\ref{lem.frag} (a), 
$\; | \{G \in \cA_{n-1}^{g(n)}: G \mbox{ connected} \}| \geqslant (1/e)\, |\cA_{n-1}^{g(n)}|\,$; and so
\[ \pr(\frag(R_n)=1) \geqslant \frac{n \, (1/e)\, |\cA_{n-1}^{g(n)}|}{|\cA_n^g|}
= \frac1{e} \cdot  \frac1{\fsgr(\cA^g,n))}.\]
This establishes the claim~(\ref{claim.conn}).
\smallskip

Let $\eps>0$.  By Lemma~\ref{lem.baddprob} there is a $c=c(\eps)$ such that
\begin{equation} \label{eqn.disconn}
\pr(R_n \mbox{ is connected}) \geqslant 1- \eps - c \, \pr(\frag(R_n)=1).
\end{equation} 
Now by~(\ref{claim.conn}) and~(\ref{eqn.disconn}),
\[ \pr(R_n \mbox{ is connected}) \geqslant 1-\eps - c \, / \, \fsgr(\cA^g,n),\]
so if $\fsgr(\cA^g,n) \to \infty$  as $n \to \infty$ then whp $R_n$ is connected. Conversely, if 
$\liminf_{n \to \infty} \fsgr(\cA^g,n) = \beta < \infty$ then $\limsup_{n \to \infty} \pr(\frag(R_n)=1) \geqslant (e \beta)^{-1} > 0$, so $\pr(R_n \mbox{ is connected}) \not \to 1$.  This completes the proof of the lemma.
\end{proof}

\begin{proof}[Proof of Theorem~\ref{thm.conn} (c)]
Suppose that $g(n) \gg n$.
By~(\ref{eqn.fsgr1}) and Lemma~\ref{lem.conn-and-fsgr}, we want to show that $\E[e(S_{n-1})] \gg n$; and to show this, it suffices to show that $\E[e(R_{n})] \gg n$ (since $\hat{g}(n) =g(n+1) \gg n$).
Let $g_1(n)$ be any genus function such that $n \ll g_1(n) \ll n^2$, say $g_1(n) = \lfloor n^{3/2} \rfloor$. Let $g_2(n) = \min \{ g(n), g_1(n) \}$, and note that $n \ll g_2(n) \ll n^2$.  By equation~(\ref{eqn.lbasymp})
\[ |\cA^g_n| \geqslant  |\cA^{g_2}_n| \geqslant \frac{n^2}{7(n+g_2(n))}\, |\cA^{g_2(n)-2}_n| \gg |\cA^{g_2(n)-2}_n|. \]
Hence whp $R_n$ is not in $\cA^{g_2(n)-2}_n$, and so by equation~(\ref{eqn.egeqg}) we have $e(R_n) \geqslant g_2(n)-1$ whp.
%
Therefore $\E[e(R_n)] \gg n$ as desired, 
and this completes the proof.
\end{proof}

We have shown that when $g(n)\gg n$, the random graph $R_n\inu\cA^g$ is connected whp. On the other hand, when $g(n)=O(n/\log n)$, it is not the case that $R_n$ is connected whp. 
We conjecture that as long as $g(n)=O(n/\log n)$, the probability that $R_n$ is connected is in fact strictly bounded away from~1.  Is there a connectivity phase transition at around $n/\log n$?
We also saw in Theorem~\ref{thm.conn} that if $g(n)$ is $o(n/\log^3 n)$ then $\,\limsup_{n \to \infty} \pr(R_n \mbox{ is connected}) \leqslant p^*$. Perhaps for \emph{every} genus function $g$ we have  $\liminf_{n \to \infty} \pr(R_n \mbox{ is connected}) \geqslant p^*$?


\section{Random graphs in a hereditary class $\hered(\cA^g)$ of embeddable graphs}
\label{sec.hered}

Recall that, given a genus function $g$, $\hered(\cA^g)$ is the class of graphs which are hereditarily in $\cA^g$.  In this section we prove Theorems~\ref{thm.heredfrag},~\ref{thm.heredconn} and~\ref {thm.heredleaves}, concerning random graphs $R_n \inu \hered(\cA^g)$. 
We begin by proving Theorem~\ref{thm.heredfrag} on the fragment and connectedness of $R_n$ for sufficiently small genus functions $g$. 
We then prove Theorem~\ref{thm.heredleaves} on the number of leaves, for three ranges of genus function $g$. Finally, we use Theorem~\ref{thm.heredleaves} to prove Theorem~\ref{thm.heredconn} on the connectedness of $R_n$, for three ranges of genus function $g$.

To prove Theorem~\ref{thm.heredfrag} we use one preliminary lemma, Lemma~\ref{lem.HFragplanar}, which shows that in the hereditary case we at least have a weakened form of the conclusion of Theorem~\ref{thm.Fragplanar}.  The proof is similar to that of Lemma~\ref{lem.nonori}.

\begin{lemma} \label{lem.HFragplanar}
Let $g(n)=o(n)$, let the set of graphs $\cG \subseteq \cE^g$  be bridge-addable, and let $R_n \inu \cG$ (with the usual understanding about $\cG_n$ being non-empty). Then whp $\Frag(R_n)$ is planar.
\end{lemma}

\begin{proof}
Let $\eps>0$.  Let $k_0 \geqslant 4/\eps$.  Then, since $\E[\frag(R_n)]<2$ by Lemma~\ref{lem.frag} (b), we have $\pr(\frag(R_n) >k_0)<\eps/2$.
Now let $K$ be any (unlabelled) connected nonplanar graph with $v(K)=k \leqslant k_0$, and let $n>k_0$. 
From each graph $G \in \cG_n$ which has a component $K$, by adding a bridge between the component and the rest of the graph, we construct $k(n-k)$ graphs $G' \in \cG_n$ which have a pendant appearance of $K$.  Observe that $G'$ can have at most $g$ disjoint pendant appearances of $K$ (since the Euler genus of a graph is the sum of the Euler genera of its blocks, see Theorem 4.4.2 of~\cite{GraphsonSurfaces}),
and for each there are at most $k$ corresponding oriented bridges pointing to the rest of the graph. Thus each graph $G'$ can be constructed at most $kg$ times.
Hence
\[ \big| \{G \in \cG_n: G \mbox{ has a component } K\} \big| \cdot k (n-k) \leqslant |\cG_n| \cdot kg, \]
so
\[ \pr(R_n \mbox{ has a component } K) \leqslant g/(n-k). \]
Let $t$ be the number of unlabelled possible connected nonplanar graphs $K$ on at most $k_0$ vertices.  Then
\[ \pr(\Frag(R_n) \mbox{ is nonplanar}) \leqslant \pr(\frag(R_n) > k_0) + t\, g/(n-k_0) < \eps \]
for $n$ sufficiently large.
\end{proof}
For the remainder of this section we always take $R_n\inu \hered(\cA^g)$.
%

\begin{proof}[Proof of Theorem~\ref{thm.heredfrag}]
Assume that $g(n) \ll n/\log^3 n$, so
$\hered(\cA^g)$ has growth constant $\gamma_{\cP}$, by Theorem~\ref{theorem:gc}~(a).
Recall that $\cC$ is the set of connected planar graphs.  For each graph $H \in \cC$, let $\alpha_H= h \rho(\cP)^h/\aut\, H$ where $h=v(H)$. Now let $H \in \cC$ and let $0 <\eps<1$.  Then as in~(\ref{eqn.pend-g}) we have
$\pend(R_n,H) \geqslant (1-\eps) \alpha_H n$ wvhp (without assuming that $g$ is non-decreasing).
Also $H$ is detachable from $\hered(\cA^{g})$, so by~(\ref{eqn.pendb}) 
we have
$\pend(R_n,H) \leqslant (1+\eps) \alpha_H n$ wvhp. Putting these together gives
\begin{equation} \label{eqn.heredpend} (1-\eps) \alpha_H n \leqslant \pend(R_n,H) \leqslant (1+\eps) \alpha_H n \;\; \mbox{ wvhp}.
\end{equation}
It follows that the conditions of Lemma 5.6  in~\cite{PendAppComp} hold for $\hered(\cA^g)$ and $\cC$ (with $\cF$ in the lemma as $\cP$).  Hence, by that lemma, $\Frag(R_n,\cP)$ converges in total variation to $BP(\cP,\rho(\cP))$ as required.
\end{proof}
Next we prove the three parts of Theorem~\ref{thm.heredleaves} on the number of leaves.  Part (a) follows immediately from the special case  of~(\ref{eqn.heredpend}) when $H$ is a single vertex.  We now prove parts (b) and (c). 

%
%
\begin{proof}[Proof of Theorem~\ref{thm.heredleaves} (b)]\m{need non-decreasing}
Suppose that $g(n)=O(n/\log n)$ and $g$ is non-decreasing. By Theorem \ref{theorem:gc} (b) and since $\hered(\cA^g)\subseteq\cA^g$, we know that $\hered(\cA^g)$ has radius of convergence $\rho(\hered(\cA^g))>0$.  Also, single vertices are attachable to $\hered(\cA^g)$ (to form leaves) since $g$ is non-decreasing, and they are detachable by the hereditary property.
Let $0<\alpha<\rho(\hered(\cA^g))$. 
Let $a_n = (|\hered(\cA^g)_n|/n!)^{1/n}$ for each $n \in \N$.  Then $\limsup_{n \to \infty} a_n = \rho(\hered(\cA^g))^{-1}$.
Let $n_1 < n_2< \cdots$ be such that $a_{n_i} \to \rho(\hered(\cA^g))^{-1}$ as $i \to \infty$. Then by Theorem~\ref{thm.app} (see also equation~(\ref{eqn.penda})), 
$\, \pr(\ell(R_{n_i}) < (1-\eps) \rho n_i) = e^{-\Omega(n_i)}$ as $i \to \infty$ and $\, \pr(\ell(R_{n_i}) > (1+\eps) \rho n_i) = e^{-\Omega(n_i)}$ as $i\to\infty$. The result of Theorem~\ref{thm.heredleaves} (b) then follows.
\end{proof}

\begin{proof}[Proof of Theorem~\ref{thm.heredleaves} (c)]
Suppose that a graph $G \in \hered(\cA^g)$ has a leaf $v$, and we form $G'$ by adding an edge incident with~$v$. The key point in the proof is that $G' \in \hered(\cA^g)$.  To see why this holds, let $H$ be an induced subgraph of $G'$, with $h$ vertices (perhaps $H=G'$). We must show that $H \in \cA^g$.  If $v$ is not in $H$ then $H$ is an induced subgraph of $G$ so $H \in \cA^g$.  Suppose that $v$ is in $H$. Then $H-v$ is an $(h\!-\!1)$-vertex graph in $\cA^g$, so $H-v$ embeds in a surface of Euler genus at most $g(h-1)$, and thus $H$ embeds in a surface of Euler genus at most $g(h-1)+2 \leqslant g(h)$ (and where the surface has the appropriate orientability); that is, $H \in \cA^g$, as required.

It is convenient to show next that we are very unlikely to have linearly many leaves. Indeed, let $0<\eta<1$.  We shall see that there is a function $t(n) \gg n$ such that $\pr(\ell(R_n) \geqslant \eta n) \leqslant e^{-t(n)}$.
Let $g_2$ be the genus function such that $g_2(n)=2n$, and for convenience write $g$ also as $g_1$.  Let $V$ be partitioned into $V_1 \cup V_2$; for $i=1,2$ let $G_i \in \hered(\cA^{g_i})$ be a graph on $V_i$;
and let $G$ be the graph $G_1 \cup G_2$ on $V$.  We claim that $G \in \hered(\cA^g)$.  To see this, let $V' \subseteq V$; and for $i=1,2$ let $V'_i =V' \cap V_i$ and let $G'_i$ be the induced subgraph of $G_i$ on $V'_i$, so we have $G'_i \in \cA^{g_i}$.
Then the induced subgraph $G'= G[V'] = G'_1 \cup G'_2$ of $G$ on $V'$ embeds in a surface of Euler genus at most
\[ g(|V'_1|) + g_2(|V'_2|) \leqslant g(|V'|),\]
and where the surface is orientable or not as required.  This completes the proof of the claim that $G \in \hered(\cA^g)$.

Let $\cB$ be the set of
graphs $G \in \hered(\cA^g)$ such that $\ell(G) \geqslant \eta\, v(G)$. From the graphs in $\cB_n$ we shall construct many graphs in $\hered(\cA^g)_n$, with limited double counting.  To avoid cluttering notation assume that $m=\eta n$ is an integer.
Given $G \in \cB_n$, pick a set $W$ of $m$ leaves in $G$, and delete the incident edges to yield the 
graph $G' \in \hered(\cA^g)_n$.  Now add to $G'$ any connected graph in $\hered(\cA^{g_2})$ on $W$, forming $G'' \in \hered(\cA^g)_n$.  The set $\hered(\cA^{g_2})_m$
of graphs is bridge-addable, so by Lemma~\ref{lem.frag} it contains at least $|\hered(\cA^{g_2})_m|/e$ connected graphs. 
Thus the number of constructions is at least $|\cB_n| \cdot |\hered(\cA^{g_2})_m|/e$. 
In $G''$ there are at most $n/m =\eta^{-1}$ components of size $m$. Thus the number of times each graph $G''$ is constructed is at most $\eta^{-1} (n-m)^m \leqslant n^{m}$ for $n$ sufficiently large. 
Then
\[ | \cB_n| \cdot |\hered(\cA^{g_2})_m|/e  \leqslant |\hered(\cA^g)_n| \cdot n^{m}\,,\]
and so
\[\pr(\ell(R_n) \geqslant \eta n) = | \cB_n|/|\hered(\cA^g)_n| \leqslant e\, n^{m} / |\hered(\cA^{g_2})_m|.\]
By \cc{Theorem 8 of~\cite{MSsizes}}
\[ f(m) := \big(|\hered(\cA^{g_2})_m|/m! \big)^{1/m} \to \infty \;\; \mbox{ as } m \to \infty\,.\]
Since $m! \geqslant (m/e)^m$,
\[ n^m/|\hered(\cA^{g_2})_m| = n^m / f(m)^m m! \leqslant \big(f(m) (m/en)\big)^{-m} = (f(\eta n) \eta/e)^{-\eta n}.\]
Thus if we let $t(n) = \log(f(\eta n)\, \eta/e) \cdot (\eta\,n) -1$, then $t(n) \gg n$ and $\pr(\ell(R_n) \geqslant \eta n) \leqslant e^{-t(n)}$.
\smallskip

We may now complete the proof.  Let $\eps>0$. 
Let $k_0=k_0(n) \sim n/\log\log n$ (say) for $n \in \N$.
By Lemma~\ref{lem.frag}~(a) we have $\kappa(R_n) \leqslant_s 1+\Po(1)$, and so there exists $t_2=t_2(n) \gg n$ 
such that $\pr(\kappa(R_n) >k_0) < e^{-t_2}$ for all $n \in \N$.
Let $n \in \N$.
For $j =0,1,2,\ldots$ let
\[ \cD^j = \{ G \in \hered(\cA^g)_n : \ell(G)=j\, \land \,\kappa(G) \leqslant k_0\}. \]
Let $j \in \N$ with $j \leqslant \eta n$. Given $G \in \cD^j$, by adding a new edge between a leaf and a vertex of degree at least 2 we can construct at least $j (n-j-k_0)$ graphs $G'$. We saw above that each such graph $G'$ is in $\hered(\cA^g)$ and so $G'$ must be in $\cD^{j-1}$.
The number of times that $G'$ can be constructed is at most twice the number of vertices of degree 2 in $G'$, which is at most $2(n-j) \leqslant 2n$.  Hence $|\cD^j|\, j(n-j-k_0) \leqslant |\cD^{j-1}|\, 2 n$; so, since $j \leqslant \eta n$ and $k_0 \ll n$, 
\[|\cD^j| \leqslant \frac{2}{j} \frac{n}{n-j-k_0}\,|\cD^{j-1}|
\leqslant \frac{2+\eps}{j}\,|\cD^{j-1}| \]
for $n$ sufficiently large (uniformly over relevant $j$).  Thus, for $n$ sufficiently large, for all integers $1 \leqslant j \leqslant \eta n$ we have 
$|\cD^j| \leqslant \frac{2+\eps}{j}\,|\cD^{kj-1}|$.
Let $X_n = {\bf 1}_{\{\ell(R_n)  \leqslant \eta n\}} {\bf 1}_{\{\kappa(R_n) \leqslant k_0\}} \ell(R_n)$.  Then $X_n \leqslant_s \Po(2+\eps)$ by for example Lemma 3.3 of~\cite{C-wba}.
Hence $\E[(X_n)_{(i)}] \leqslant (2+\eps)^i$ for each $i \in \N$. 
But
\[ (\ell(R_n))_{(i)} \leqslant (X_n)_{(i)} + n^i {\bf 1}_{\{\ell(R_n)>\eta n)\}} + n^i{\bf 1}_{\{\kappa(R_n) > k_0\}},\]
so
\[ \E[(\ell(R_n))_{(i)}] \leqslant \E[(X_n)_{(i)}] + n^i \pr(\ell(R_n)>\eta n) + n^i \pr(\kappa(R_n)>k_0) \leqslant (2+\eps)^i +o(1).\]
This completes the proof of part (c) of Theorem~\ref{thm.heredleaves}, and thus of the whole theorem.
\end{proof}

Finally in this section we prove Theorem~\ref{thm.heredconn} on the connectedness of $R_n$ for three ranges of genus function~$g$ (using Theorem~\ref{thm.heredleaves}).

\begin{proof}[Proof of Theorem~\ref{thm.heredconn} (a)]
\m{need non-decreasing} 
Let $g(n) = O(n/\log n)$ and let $g$ be non-decreasing.   We must show that
\begin{equation} \label{eqn.less}
\limsup_{n \to \infty} \, \pr(R_n \mbox{ is connected}) < 1\,.
\end{equation}
Since $\hered(\cA^g)$ is a subclass of $\cA^g$ and $\rho(\cA^g)>0$ by Theorem \ref{theorem:gc} (b), we also have that $\rho :=\rho(\hered(\cA^g))>0$. Let $\alpha=\rho/2$. 
Since $g$ is non-decreasing, by Theorem~\ref{thm.heredleaves} (b), there is an infinite sequence $n_1<n_2< \cdots$ such that as $i \to \infty$ with whp $\ell(R_{n_i}) \geqslant \alpha n_i$. 
Hence, arguing as in the proof of Lemma~\ref{lemma.connectedLD}, $\pr(R_{n_i} \mbox{ is disconnected}) \geqslant \alpha/3 +o(1)$,
which completes the proof of part (a).
\end{proof}

\begin{proof}[Proof of Theorem~\ref{thm.heredconn} (b)] 
Let $g(n) \gg n/\log n$.  We must show that
\begin{equation} \label{eqn.equals1}
\limsup_{n \to \infty} \, \pr(R_n \mbox{ is connected}) =1.
\end{equation}
By \cc{Theorem 8} 
in the companion paper \cite{MSsizes}
\begin{equation} \label{eqn.hered1} \left(|\hered(\cA^g)_n|/n!\right)^{1/n} \to \infty \; \mbox{ as } n \to \infty.
\end{equation}
For each $j \geqslant 3$, each graph $G$ in $\hered(\cA^g)_j$ with $\frag(G)=1$ consists of an isolated vertex $v \in [j]$ and a connected graph $G'$ in $\hered(\cA^g)$ on vertex set $[j] \backslash \{v\}$ (and conversely each such pair $v, G'$ yields a graph $G$ in $\hered(\cA^g)_j$ with $\frag(G)=1$ if $g(j) \geqslant g(j-1)$).
Thus the probability that $\frag(R_j)=1$ is at most $j$ times the number of connected graphs in $\hered(\cA^g)_{j-1}$ divided by $|\hered(\cA^g)_j|$ (with equality if $g(j) \geqslant g(j-1)$).  Therefore
\begin{equation}\label{heredfsgr} \pr(\frag(R_j)=1) \leqslant \frac{j \: |\hered(\cA^g)_{j-1}|}{|\hered(\cA^g)_j|}. \end{equation}
Hence, for each $n_0 \geqslant 2$
\[ \prod_{j=n_0+1}^{n} \pr(\frag(R_j)=1) \leqslant \frac{n!}{n_0!} \; \frac{|\hered(\cA^g)_{n_0}|}{|\hered(\cA^g)_n|}
\]
so
\[ \left(\prod_{j=n_0+1}^{n} \pr(\frag(R_j)=1)\right)^{\frac1{n-n_0}} \leqslant \left(\frac{|\hered(\cA^g)_{n_0}|}{n_0!}\right)^{\frac1{n-n_0}} \; \left(\frac{n!}{|\hered(\cA^g)_n|}\right)^{\frac1{n-n_0}}.\]
As $n \to \infty$ (with $n_0$ fixed), the first term on the right side tends to 1, and by~(\ref{eqn.hered1}) the second term tends to~0.  But the minimum value of $\pr(\frag(R_j)=1)$ over $n_0<j \leqslant n$ is at most the term on the left side, so this minimum value tends to 0 as $n \to \infty$.
It follows that
\[ \liminf_{n \to \infty} \, \pr(\frag(R_n)=1) =0.\]
Since $\hered(\cA^g)$ is bridge-addable and closed under edge-deletion, by Lemma~\ref{lem.baddprob}
\[ \limsup_{n \to \infty} \, \pr(R_n \mbox{ is connected}) =1.\]
Thus~(\ref{eqn.equals1}) holds, which completes the proof of part (b).
\end{proof}
 
\begin{proof}[Proof of Theorem~\ref{thm.heredconn} (c)] Finally, let $g(n+1) \geqslant g(n)+2$ for each $n \in \N$.  Let $\eps>0$. By Theorem~\ref{thm.heredleaves} (c) we have $\E[\ell(R_n)] \leqslant 2+o(1)$, and so $\pr(\ell(R_n) \geqslant 5/\eps) \leqslant \eps/2$ for $n$ sufficiently large.  
Let $\cB$ be the set of graphs $G \in \hered(\cA^g)$ with $\ell(G) < 5/\eps$ and with $\frag(G)=1$. Let $n \geqslant 3$.  From each graph $G \in \cB_n$, by adding an edge incident with the isolated vertex  we can construct $n-1$ graphs $G' \in \hered(\cA^g)_n$ with $\ell(G') < 5/\eps +1$, and each graph $G'$ can be constructed at most $\ell(G')$ 
times.  Hence
\[ |\cB_n|/\big|\hered(\cA^g)_n\big| < (5/\eps +1)/(n-1).\]
Thus
\[ \pr(\frag(R_n)=1) \leqslant \eps/2 + o(1) < \eps \]
for $n$ sufficiently large.  It now follows from Lemma~\ref{lem.baddprob} that $R_n$ is connected whp, as required.
\end{proof}

It would be interesting to know whether any phase transitions in the hereditarily embeddable case behave differently to the embeddable case. For example, it is conceivable that there is a  threshold for some property in a hereditarily embeddable class which occurs around $g(n) \approx n/\log n$, whilst the corresponding threshold in the embeddable class does not occur until around $n$. This is conceivable since hereditary classes are closed under removing vertices, while general classes are not. Theorem~\ref{thm.heredconn}
for example suggests a phase transition for connectedness at $g(n) \approx n/\log n$ in the hereditary case, while no result is known between $n/\log n$ and $n$ in the embeddable case.


\section{Random graphs in a minor-closed class $\minor(\cA^g)$ of embeddable graphs}\label{minorclosed}
Let us now insist that each minor of our graphs (rather than each induced subgraph) is appropriately embeddable.  Recall that a graph $H$ is a \emph{minor} of a graph $G$ if $H$ can be obtained from a subgraph of $G$ by a sequence of edge-contractions, see for example~\cite{BondyMurty,Diestel}.
Given a class $\cB$ of graphs, let $\minor(\cB)$ be the class of graphs $G$ such that each minor of $G$ is in $\cB$. Thus $\minor(\cB)$ is minor-closed: we call it the \emph{minor-closed part} of $\cB$.
Of course we always have $\cP \subseteq \minor(\cA^g) \subseteq \cA^g$, and so in particular
$\rho(\cP) \geqslant \rho(\minor(\cA^g))$.

Observe that $\cA^g$ contains all graphs if and only if $\cA^g$ contains each complete graph $K_n$, if and only if $\minor(\cA^g)$ contains all graphs.
Also $\minor(\cA^g)$ either contains all graphs or has positive radius of convergence, see \cc{Theorem 9 of \cite{MSsizes}}. Using this result, we can obtain some results on properties of random graphs from such a class. Note that while some of the arguments we use earlier in this paper to get results in the embeddable or hereditarily embeddable cases do
transfer to the more restricted setting of minor-closed classes, others do not. For example the arguments we use to get bounds on the numbers of edges and faces do not transfer to the minor-closed setting. The same holds for the arguments to get upper bounds on the maximum degree and maximum face size.
Given $c>0$ and $0<\eps<1$, we let $(1\pm \eps)\, c$ denote the open interval $((1-\eps)c, (1+\eps)c)$ in the real line.

In all the theorems below, $g$ is a non-decreasing genus function and $R_n \inu \minor(\cA^g)$.  Observe that for the parts (a) of these theorems, where we assume that $\cA^g$ contains all graphs and thus $\minor(\cA^g)$ contains all graphs, nothing else matters about $g$ (and in particular it does not matter whether or not $g$ is non-decreasing); and in this case $R_n$ has the distribution of the binomial 
random graph $G(n,\frac12)$ where edges appear independently with  probability~$\tfrac12$, that is $R_n \sim G(n,\frac12)$.

We begin by considering the connectedness of $R_n$, and find that when the radius of convergence drops to zero, the graph $R_n$ is also with very high probability connected.

\begin{theorem}\label{thm.minorconnected}
\item{(a)}
If $\cA^g$ contains all graphs then  $R_n$ is connected \wvhp.

\item{(b)} If $\cA^g$ does not contain all graphs, then
\begin{equation}
    \liminf_{n\rightarrow \infty}\, \mathbb{P}(R_n \mbox{ is connected}) <1.
\end{equation}
\end{theorem}
\begin{proof}
For part (a), when $R_n \sim G(n,\frac12)$, we may take $p=\tfrac{1}{2}$ in the proof of Theorem~7.3 in \cite{BelaB}.
%
For part (b) note that the proof of Theorem~\ref{thm.conn} (b), as well as the proof of Theorem~\ref{thm.leaves} (b), still apply when the class $\cA^g$ is replaced by $\minor(\cA^g)$.
\end{proof}

Next, we consider the number of leaves in $R_n$. When the radius of convergence of the exponential generating function of $\minor(\cA^g)$ drops to zero, $R_n$ goes from having linearly many leaves whp to having no leaves at all wvhp.

\begin{theorem}\label{thm.minorleaves}
\item{(a)} If $\cA^g$ contains all graphs then $\ell(R_n)=0$ \wvhp.
\item{(b)} If $\cA^g$ does not contain all graphs, then for any $0<\eps<1$
\begin{equation}
    \rho(\{G \in \cB : \,\ell(G) \not\in (1 \pm \eps) \rho(\cB) v(G)\,\}) >  \rho(\cB) \,. 
\end{equation}
\end{theorem}
\begin{proof}
Consider part (a), when $R_n \sim G(n,\frac12)$.
The probability that a vertex is a leaf is
$(n-1)2^{-(n-1)}$, and so the probability that $R_n$ has no leaves is at least
\begin{equation}
\mathbb{P}(\ell(R_n)=0) \geqslant 1-n(n-1)2^{-(n-1)} > 1-n^22^{-(n-1)}.
\end{equation}
Thus $\wvhp$ $R_n$ has no leaves, as required.

For part (b), note that since $g$ is non-decreasing, leaves can be attached and detached from $\minor(\cA^g)$. Thus when $\rho(\minor(\cA^g))>0$ we can apply Theorem~\ref{thm.app}, from which the result follows directly.
\end{proof}

We now consider the maximum degree of $R_n$. When $\cA^g$ contains all graphs, $R_n$ has degree about $n/2$ with high probability. Currently, we only have a lower bound on the maximum degree when $\cA^g$ does not contain all graphs.

\begin{theorem}\label{thm.minormaxdeg}
\item{(a)} If $\cA^g$ contains all graphs then for any $\eps>0$ we have $\tfrac12 \,n < \Delta(R_n) < (\tfrac{1}{2}+\eps)\,n$ \wvhp\,.
\item{(b)} If $\cA^g$ does not contain all graphs, then there for any $\eps>0$ exists a constant $c>0$
such that
\begin{equation}
    \limsup_{n\rightarrow \infty}\, \mathbb{P}(c \log n \leqslant \Delta(R_n)) > 1-\eps.
\end{equation}
\end{theorem}
\begin{proof}
Part (a), when $R_n \sim G(n,\frac12)$, follows directly from \cite{MaxDegreeER}.
In particular, $\mathbb{P}(\Delta(R_n)\leqslant\tfrac{1}{2}n) \approx (0.6102 +o(1))^n$.
For part (b) note that the proof of the lower bound in Theorem~\ref{thm.Delta2} (b) still applies when considering the minor-closed subclass $\minor(\cA^g)$ instead of $\cA^g$.
\end{proof}

Finally, we consider the maximum face size in a relevant embedding of $R_n$. We see that when $\cA^g$ contains all graphs, then whp there exists a relevant embedding with face size of order $n^2$. When $\cA^g$ does not contain all graphs, we currently only have a lower bound.
\begin{theorem}\label{thm.minorfacesize}
\item{(a)} If $\cA^g$ contains all graphs then, for any  $0<\eps<1$, \whp\ there exists  an embedding of $R_n$ in an orientable and in a non-orientable surface of Euler genus at most $g(n)$ which has a face of size at least $(1-\eps)\tfrac18n^2$ (and with a connected boundary walk).
\item{(b)} If $\cA^g$ does not contain all graphs then, for any  $0<\eps<1$, there exists a constant $c>0$ such that the following holds. For $n\in\mathbb{N}$ let $p_n$ be the probability that, in every embedding of $R_n$ in a surface of Euler genus at most $g(n)$, the maximum face size is at least $c\log n$. Then $\limsup_{n\rightarrow \infty} p_n\geqslant 1-\eps$.
\end{theorem}
\begin{proof}
Consider part (a), when $R_n \sim G(n,\frac12)$.  By Theorem~\ref{thm.minorconnected} (a) we may assume that $R_n$ is connected.
Let $\eps > 0$. By Theorem 1.1 of \cite{GenusRandom1}, whp $R_n$ has an embedding in an orientable surface of Euler genus at most $(1+\tfrac1{2}\eps)\tfrac1{12}n^2$.
Since $\minor(\cA^g)$ contains all graphs, it must also contain all complete graphs, so  $g(n)$ is at least as large as the Euler genus of the complete graph $K_n$ for each $n \in \N$, in particular $g(n)\geqslant \left\lceil \tfrac16(n-3)(n-4) \right\rceil \geqslant (1-\tfrac1{4}\eps)\tfrac16n^2 +3$ for $n$ sufficiently large. Start with a cellular embedding of $R_n$ in an orientable surface of Euler genus at most $(1+\tfrac1{2}\eps)\tfrac1{12}n^2$, which exists whp. For $n$ sufficiently large we can add at least 
\[\left\lfloor \tfrac12((1-\tfrac1{4}\eps)\tfrac16n^2 +2-(1+\tfrac1{2}\eps)\tfrac1{12}n^2)\right\rfloor = 1+ \left\lfloor (1-\eps)\tfrac1{24}n^2 \right\rfloor\geqslant (1-\eps)\tfrac1{24}n^2\]
handles to the surface to create an embedding of $R_n$ in an orientable surface of Euler genus at most $g(n)-1$. 
We add a handle between two adjacent faces $F_1$ and $F_2$ (sharing at least one edge), which merge to form a new face $F_3$ with a single boundary walk of length the sum of the lengths of the boundary walks of $F_1$ and $F_2$; then we add a handle between $F_3$ and an adjacent face, forming a larger new face $F_4$; then add a handle between $F_4$ and an adjacent face, and so on, until either all faces are merged into one big face, or no handles are left. In the former case, since $R_n$ has at least $(1-\eps)\tfrac14n^2$ edges whp, we construct a face of size at least $(1-\eps)\tfrac12n^2$ whp. In the latter case, since the minimum face size is at least 3, we construct whp a face of size at least
\begin{equation}
 3\cdot (1-\eps)\tfrac1{24}n^2 =(1-\eps)\tfrac1{8} n^2\,,
\end{equation}
as required.  Finally, we can add a cross-cap to obtain an embedding as required in a non-orientable surface of Euler genus at most $g(n)$.
\smallskip

For part (b) note that the proof of the lower bound in Theorem~\ref{thm.facesize} (b) still applies when considering the minor-closed subclass $\minor(\cA^g)$  instead of $\cA^g$.
\end{proof}

\section{Random unlabelled graphs $\tilde{R}_n$} \label{sec.unlab}

For a given genus function $g$, we consider random unlabelled graphs
$\tilde{R}_n \inu \tA^g_n$,
$\tilde{R}_n \inu \hered(\tA^g)$, and $\tilde{R}_n \inu \minor(\tA^g)$.
We investigate the probability that $\tilde{R}_n$ is connected, and we prove Theorem~\ref{thm.unlab}. We can say little at present compared with Theorem~\ref{thm.conn}, on the probability of being connected in the labelled case.
We also prove the following theorem, which concerns a minor-closed class, and which gives a result for $\tilde{R}_n \inu \minor(\tA^g)$ corresponding closely to Theorem~\ref{thm.minorconnected} in the labelled case.
\begin{theorem}\label{thm.minorconnectedU}
Let $g$ be a genus function, and let $\tilde{R}_n \inu \minor(\tA^g)$.
\begin{description}
\item{(a)} If $\tA^g$ contains all graphs, then $\tilde{R}_n$ is connected \wvhp\, (and indeed $\pr(\tilde{R}_n \mbox{ not connected}) \sim n\,2^{-n+1}$).

\item{(b)} If $g$ is non-decreasing and $\tA^g$ does not contain all graphs, then
\begin{equation}
    \liminf_{n\rightarrow \infty}\, \mathbb{P}(\tilde{R}_n \mbox{ is connected}) <1.
\end{equation}

\end{description}
\end{theorem}
In part (a) the fact that $\tilde{R}_n$ is connected whp (rather than wvhp) follows directly from Theorem 9.5 of~\cite{BelaB}.
Both Theorem~\ref{thm.unlab} and Theorem~\ref{thm.minorconnectedU} (b) will follow from Theorem~\ref{thm.unlabnew}, which gives an asymptotic lower density version of these results.
The following lemma is the main step in the proofs.

\begin{lemma} \label{lem.unlabnew}
Let the set $\tB$ of unlabelled graphs be closed under adding an isolated vertex, let $\tilde{\rho}(\cB) >0$, and let $\tilde{R}_n \inu \tB$.  
Then for any $\eps>0$ there is a set $I^* \subseteq \N$ with lower 
density at least $1-\eps$ such that $\sup_{n \in I^*} \pr(\tilde{R}_n \mbox{ is connected}) <1$.
\end{lemma}
\begin{proof}
Let $\tD$ denote the set of disconnected graphs in $\tB$. For each $n \geqslant 2$ 
\begin{equation} \label{eqn.ugrow0}
    |\tB_{n}| \geqslant |\tD_{n}| \geqslant |\tB_{n-1}|.
\end{equation}
To prove the second inequality in~(\ref{eqn.ugrow0}), just consider adding an isolated vertex to each graph in $\tB_{n-1}$.  Let $\alpha = \tilde{\rho}(\cB)^{-1}$.
Let $n_0$ be large enough that $|\tB_n|^{1/n} \leqslant \alpha +1$ for each $n \geqslant n_0$.
Let $\beta = (\alpha +1)^{1/\eps}$. 
Let $J^*=J^*(g,\eps)$ be the set of $j \in \N$ such that either $j=1$, or $j \geqslant 2$ and $ |{\tB}_{j}| \leqslant \beta\, |{\tB}_{j-1}|$.

Let us show that for each $n \geqslant n_0$, $|J^* \cap [n]| \geqslant (1-\eps)n$.
Let $n \geqslant n_0$ and suppose for a contradiction that $|J^* \cap [n]| < (1-\eps)n$; that is, $|{\tB}_{j}| > \beta\, |{\tB}_{j-1}|$
for more than $\eps n$ values $j \in \{2,\ldots,n\}$.  But then by~(\ref{eqn.ugrow0})
\[ |\tB_n|^{1/n} > (\beta^{\eps n})^{1/n} = \alpha  +1 \]
contradicting our choice of $n \geqslant n_0$.  Thus $J^*$ has lower density at least $1-\eps$.

Now let $n \geqslant 2$ be in $J^*$.  Then
$ |{\tB}_{n}| \leqslant \beta\, |{\tB}_{n-1}|$, and $ |{\tD}_{n}| \geqslant |{\tB}_{n-1}|$ by~(\ref{eqn.ugrow0}), so
\[ \pr(\tilde{R}_n \mbox{ is connected}) = 1-|{\tD}_{n}|/|{\tB}_{n}| \leqslant 1-1/\beta, \]
which completes the proof.
\end{proof}

\begin{theorem} \label{thm.unlabnew}
Let the genus function $g$ be non-decreasing.
Suppose that either (a) $g(n) = O(n/\log n)\,$ and either $\tilde{R}_n \inu \tilde{\cA}^g$ or $\tilde{R}_n \inu \hered(\tA^g)$; or (b) $\cA^g$ does not contain all graphs and $\tilde{R}_n \inu \minor(\tA^g)$.  Then for any $\eps>0$ there is a set $I \subseteq \N$ with lower 
density at least $1-\eps$ such that \[ \sup_{n \in I^*} \pr(\tilde{R}_n \mbox{ is connected}) <1\,, \]
and a fortiori
\[ \liminf_{n \to \infty}\, \pr(\tilde{R}_n \mbox{ is connected}) <1.\]
\end{theorem}

\begin{proof}\m{need non-decreasing}
Since $g$ is non-decreasing, each of the relevant classes of graphs is closed under adding an isolated vertex. If $g(n) = O(n/\log n)$ then by~(\ref{eqn.rhopos}) we have $\tilde{\rho}(\hered(\tA^g)) \geqslant \tilde{\rho}(\tA^g) >0$.
If $\cA^g$ does not contain all graphs then $\rho(\minor(\tA^g))>0$ by \cc{Theorem 9 of~\cite{MSsizes}.} Theorem~\ref{thm.unlabnew} now follows from Lemma~\ref{lem.unlabnew}.
\end{proof}

Both Theorem~\ref{thm.unlab} and Theorem~\ref{thm.minorconnectedU} (b) follow from Theorem~\ref{thm.unlabnew}.  It remains here to prove Theorem~\ref{thm.minorconnectedU}~(a).

\begin{proof}[Proof of Theorem~\ref{thm.minorconnectedU} (a)]
Let $\ell_n$ be the number of graphs on vertex set $[n]$, so $\ell_n = 2^{\tbinom{n}2}$.
Let $u_n$ be the number of unlabelled $n$-vertex graphs.  Clearly $\ell_1=u_1=1$, and $\ell_n \leqslant u_n \, n!$ for all $n \in \N$. It was shown by Poly\'a that $u_n \sim \ell_n/n!$, see for example~\cite[Page 105]{FSbook}, and see also~\cite[Section 9.1]{BelaB}. It follows that there is a constant $c$ such that $u_n \leqslant c\, \ell_n/n!$ for all $n \in \N$. 
Now
\[ \pr(\tilde{R}_n \mbox{ not connected}) \geqslant \frac{u_1 \, u_{n-1}}{u_n} \sim \frac{n \,\ell_{n-1}}{\ell_n} = n\,2^{-n+1}.\]
On the other hand, 
\begin{eqnarray*}
\pr(\tilde{R}_n \mbox{ not connected}) & \leqslant &
\sum_{k=1}^{\lfloor n/2\rfloor} \frac{u_k u_{n-k}}{u_n} \; = \frac{u_1\,u_{n-1}}{u_n} + \sum_{k=2}^{\lfloor n/2\rfloor} \frac{u_k u_{n-k}}{u_n} \\
& \leqslant &
(1+o(1))\, n\,2^{-n+1} +
c^2 \, \sum_{k=2}^{\lfloor n/2\rfloor} \binom{n}{k} \frac{\ell_k \,  \ell_{n-k}}{\ell_n}.
\end{eqnarray*}
But for $2 \leqslant k \leqslant n/2$
\[ \frac{\ell_k \,  \ell_{n-k}}{\ell_n} \leqslant \frac{\ell_2 \,  \ell_{n-2}}{\ell_n} = 2^{-2n+4}\,, \]
so
\[ \sum_{k=2}^{\lfloor n/2\rfloor} \binom{n}{k} \frac{\ell_k \,  \ell_{n-k}}{\ell_n} \leqslant \sum_{k=2}^{\lfloor n/2\rfloor} \binom{n}{k} \cdot 2^{-2n+4}  \leqslant 2^{n-1} \cdot 2^{-2n+4} = 2^{-n+3}\,; \]
and thus
$\,\pr(\tilde{R}_n \mbox{ not connected}) \leqslant
(1+o(1))\, n\,2^{-n+1}$.
Hence
\[ \pr(\tilde{R}_n \mbox{ not connected}) \sim n\,2^{-n+1},\]
as required.
\end{proof}

We have seen that in the unlabelled graph classes, when $g(n)=O(n/\log n)$ and $g$ is non-decreasing, the 
probability that $\tilde{R}_n$ is connected is bounded away from 1. 
We have focussed on connectedness, but it would be interesting to know more about typical properties of $\tilde{R}_n$. It is worth noting, however, that even in the planar case, little is known about typical properties of the random unlabelled planar graph. In \cite{BFKV2007} the authors present several results on properties such as connectedness, number of edges and chromatic number of the random unlabelled outerplanar graph.

\section{Concluding remarks and questions}
\label{sec.concl}

Recall that $g$ is a genus function and $\cA^g$ is either $\cO\cE^g$, $\cN\cE^g$, $\cE^g$ or $\cO\cE^g\cap \cN\cE^g$. For random graphs in $\cA^g$, we have given estimates and bounds on the probability of being connected, the typical numbers of leaves, edges and faces, the typical maximum degree, and the typical maximum size of a face. Further, we have given some corresponding results for the hereditary case and the minor-closed case, and for random unlabelled graphs. Many interesting questions remain open. While we have managed to find some answers which differ between genus $o\left(n/\log n \right)$ and genus $\omega(n)$, for example in Theorem~\ref{thm.conn} on connectivity or Theorem~\ref{thm.leaves} on the number of leaves,
we have not for example found answers which differ between genus $o\left(n/\log n \right)$ and genus $\omega(n/\log n)$, apart from in Theorem~\ref{thm.heredconn} on connectivity in the hereditary case. This raises the question of whether phase transitions for these properties occur when $g(n)$ is around $n/\log n$, or around $n$, or whether there is even a slow transition happening between $n/\log n$ and $n$. Some further concluding remarks and conjectures on connectivity, number of edges and faces, number of leaves, and maximum degree and face size were given at the end of Section~\ref{sec.conn}, Section~\ref{sec.edgesfaces}, Section~\ref{sec.leaves} and Section~\ref{degfacesize_LDversions}, respectively.

In the hereditarily embeddable case we have presented results on the fragment, connectivity and number of leaves. While these results are stronger than in the embeddable case, we would like to learn about other properties. 
Further, it would be interesting to know whether phase transitions in the hereditarily embeddable case behave differently from the embeddable case. For example, it is conceivable that there is a threshold for some property in a hereditarily embeddable class when $g(n)$ is around $n/\log n$, whilst the corresponding threshold in the embeddable class does not occur until later. 


\bibliographystyle{abbrv}
\bibliography{main}

\end{document}